\documentclass[12pt,a4paper]{article}%
\usepackage{amsmath}
\usepackage{graphicx}
\usepackage{epsfig}%
\usepackage{amsfonts}%
\usepackage{amssymb}
\usepackage{algorithm}
\usepackage{algorithmic}
\usepackage{color}

\newtheorem{theorem}{Theorem}[section]

\newtheorem{definition}[theorem]{Definition}

\newtheorem{lemma}[theorem]{Lemma}

\newtheorem{remark}[theorem]{Remark}

\newenvironment{proof}[1][Proof]{\noindent \emph{#1.} }{\hfill \
\rule{0.5em}{0.5em}}

\makeatletter\@addtoreset{equation}{section}\makeatother
\makeatletter\@addtoreset{figure}{section}\makeatother
\makeatletter\@addtoreset{table}{section}\makeatother
\begin{document}
 
 \title{Numerical study in stochastic homogenization for elliptic PDEs:
 convergence rate in the size of  representative  volume elements}

\author{Venera Khoromskaia \thanks{Max Planck Institute for
        Mathematics in the Sciences, Leipzig;
        Max Planck Institute for Dynamics of Complex Technical Systems, Magdeburg ({\tt vekh@mis.mpg.de}).}
        \and
        Boris N. Khoromskij \thanks{Max-Planck-Institute for
        Mathematics in the Sciences, Inselstr.~22-26, D-04103 Leipzig,
        Germany ({\tt bokh@mis.mpg.de}).}
        \and
        Felix Otto \thanks{Max-Planck-Institute for
        Mathematics in the Sciences, Inselstr.~22-26, D-04103 Leipzig,
        Germany ({\tt otto@mis.mpg.de}).}
        }

 \date{}

\maketitle

\begin{abstract}

We describe the numerical scheme for the discretization and solution of 2D elliptic equations
with strongly varying piecewise constant coefficients arising in the stochastic homogenization
of multiscale composite materials.
An efficient stiffness matrix generation scheme based on assembling the local Kronecker
product matrices is introduced.
The resulting large linear systems of equations are solved by the preconditioned 
CG iteration  with a convergence rate that is independent of the grid size and 
the variation in jumping coefficients (contrast).  
Using this solver we numerically investigate the convergence of the 
Representative Volume Element (RVE) method in stochastic homogenization that 
extracts the effective behavior of the random coefficient field.
Our numerical experiments confirm the asymptotic convergence rate  
of systematic error and standard deviation in the size of RVE rigorously established 
in \cite{GlNeuOtto:13}. The asymptotic behavior of covariances of the homogenized matrix 
in the form of a quartic tensor is also studied numerically.
Our approach allows laptop computation of sufficiently large number 
of stochastic realizations even for large sizes of the RVE.  
\end{abstract}



\section{Introduction}\label{Int:SH} 

Homogenization methods allow to derive the effective mechanical and physical
properties of highly heterogeneous  materials from the knowledge of the spatial distribution of their
components \cite{BrisLegoll:11,KaFoGa:03}.   In particular,  
stochastic homogenization via the Representative Volume Element (RVE) methods  
provide means for calculating the effective large-scale characteristics related to
structural and geometric properties of random composites,
by utilizing a possibly large number of probabilistic 
realizations \cite{GlOtto:16,GlOtto:15,GlOtto:12,GlNeuOtto:13,Fischer:18}.  
 The numerical investigation of the effective characteristics of random structures 
is a challenging problem since the underlying elliptic equation with randomly generated 
coefficients should be solved for many 
thousands realizations and for domains with substantial structural complexity to obtain 
sufficient statistics. Note that for every realization of the random medium
one should construct a new  stiffness matrix and right hand side to solve the discretized problem.
Therefore, construction of  fast solvers 
(which allow to confirm numerically the quantitative results in the 
stochastic homogenization) using the conventional computing facilities is a challenge.

This paper  presents the   numerical study of 
the stochastic homogenization  of an elliptic system with randomly generated coefficients.
Our approach is based on the FEM-Galerkin approximation of the 2D elliptic equations 
in a periodic setting by using fast assembling of the FEM stiffness matrix 
in a sparse matrix format, which is performed by agglomerating the Kronecker tensor products of 
simple 1D FEM discrete operators   \cite{KKO:17}.  
We use the product piecewise linear finite elements on the rectangular grid assuming that  
 strongly varying piecewise constant equation coefficients are resolved on that grid. 
This scheme provides efficient approximation of equations with  complicated   
 jumping coefficients.  The numerical analysis of the error in the Galerkin FEM approximation 
  indicates the convergence
rate $O(h^{\beta})$ in the $L^2$-norm with $3/2 \leq \beta \leq 2$.

The resulting large linear system of equations is solved by the preconditioned  
CG iteration, where the convergence rate is proven to be independent on the grid size and  
the relative variation in jumping coefficients, i.e., on the contrast.
The preconditioned iterative solvers for the discrete  elliptic systems of equations 
with  variable coefficients have been long since 
in the focus in the literature on numerical methods for multi-dimensional and stochastic PDEs,  
see \cite{KhWi:B,KhorBook:17,BoSchw:11} and the literature therein.

 In this paper, we consider an ensemble of two-valued random coefficient fields, 
  which is based on independently and uniformly placed (and thus overlapping) 
  axi-parallel square inclusions of fixed side length.
We investigate the RVE method that
(approximately) extracts the effective (i.e. large-scale) behavior
of the medium in form of the deterministic and homogeneous matrix $\mathbb{A}_{hom}$ 
from a given (stationary and ergodic) ensemble.
This method produces an approximation to $\mathbb{A}_{hom}$ by solving $D=2$ elliptic 
equations on a torus of (lateral) size $L$ with a specific right hand side (the 
corrector equation), by taking the spatial average of the flux of these solutions, 
and by taking the empirical mean over $N$ independent realizations of this coefficient 
field under the naturally periodized version of the ensemble.
This is an approximation in so far as the outcome is still random (as quantified by 
the standard deviation of the outcome of a
single realization) and that the periodic boundary conditions affect the statistics 
(which we call the systematic error).
In \cite{GlNeuOtto:13}, Gloria, Neukamm and the last author rigorously derived upper bounds how the 
standard deviation and the systematic error decrease with increasing RVE size $L$.
Our numerical experiments confirm the scaling of these bounds.
Since numerically, there is no access to exact values of the variance (or standard expectation) 
or the expectation, we replace these quantities by their empirical counterparts 
for a large number of realizations $N$.
We thus first provide numerical evidence that these quantities have saturated in $N$, 
and second that their limiting values display the predicted scaling in $L$. 

In work \cite{GlOtto_Fluct:17} by Duerinckx, Gloria and the last author, it was worked out that 
the properly rescaled variance of the output of the RVE converges as $L\uparrow\infty$ 
to a quartic tensor $Q$ that governs the leading-order fluctuations of any solution. 
In this paper, we show how the symmetry properties of the ensemble yields symmetry 
properties of $Q$ (and its approximation). Also, a convergence rate was rigorously 
established in that work, and is being numerically investigated here. While the 
number of samples $N$ considered here seems too low to reach saturation 
in the empirical variance, the findings are at least not in contradiction with 
the theoretic results.  
 
In numerical tests on the stochastic properties of the 2D RVE method  
we study the asymptotics of empirical variance versus the size of RVE $L\leq 128$,
and  of the systematic error versus the number of realizations $N$ up to $N=10^{5}$.
Furthermore, we estimate the convergence of the quartic  tensor by 
implementing a large number of stochastic realizations.
The proposed  techniques allow to compute  a sufficiently large number 
of  realizations of random coefficient fields with  
a large  number of overlapping inclusions up to $L^2=128^2$   
corresponding to the stiffness matrix size $513^2\times 513^2$ 
using  MATLAB on a moderate computer cluster.

The numerical investigation of the stochastic homogenization problem attracts interest
and becomes an active field of research, see the survey \cite{BrisLegoll:11} and references
therein. Recently the numerical solution of the corrector-type problem, in the context of 
homogenization of the diffusion equation with spherical inclusions 
 by using boundary element methods  and 
 the fast multipole techniques has been considered in \cite{CaEhrLeStXi:18}.

The rest of the paper is organized as follows. In Section \ref{Int2:SH}, we address the 
problem setting and define the elliptic equations of stochastic homogenization. 
Section \ref{sec:Matr_Solv} describes the Galerkin-FEM discretization scheme based on the fast 
matrix generation by using sums of Kronecker products of single-dimensional 
matrices. We also outline the PCG iteration  applied in the computer simulations
and provide numerics on the FEM discretization error. 
Section \ref{sec:Comp_Ahom} introduces the 
computational scheme for the stochastic average coefficient matrix.
Furthermore, in Section \ref{ssec:Symmetr_Quartic_tensQ} we describe the construction 
and properties of the covariances of the homogenized matrix in the form of a quartics tensor.
Section \ref{sec:Homo_Numer} presents results of numerical experiments 
on the  empirical average and systematic error at the limit of a large 
number of stochastic realizations. 
The asymptotic of the  quartic tensor versus the leading order variances 
is analyzed numerically in Section \ref{ssec:Quartic_tensQ}.
Conclusions outline the main results of the paper. 
In Appendix the spectral properties of the discrete stochastic elliptic operators  
  are studied numerically for a sequence of stochastic realizations.

\section{Elliptic equations in stochastic homogenization }\label{Int2:SH}

In this section, we describe the problem setting in the stochastic homogenization theory. 
For given $f\in {\cal L}^2(\Omega)$ such that $\int_{\Omega} f(x) dx =0$,
we consider the class of model elliptic boundary value problems on the $d$-dimensional 
torus $\Omega :=[0,L)^d$ of side-length $L\in \mathbb{N}$: find $\phi\in H^1(\Omega)$, s.t.
\begin{equation} \label{eqn:hm_setting}
  {\cal A}\phi := -\nabla \cdot \mathbb{A}(x)\nabla \phi =f(x), \quad 
x=(x_1,\ldots,x_d)\in \Omega,
\end{equation}
with the diagonal $d\times d$ uniformly elliptic coefficient matrix $\mathbb{A}(x)$, 
$\infty > \beta_0 I \geq\mathbb{A}(x)\geq\alpha_0 I>0$. 
In this paper we focus on the special class of
elliptic problems (\ref{eqn:hm_setting}) arising in stochastic homogenization theory where 
the highly varying coefficient matrix and the right-hand side are defined 
by a sequence of stochastic realizations
as described in \cite{GlOtto:16,GlOtto:15,GlOtto:12,GlNeuOtto:13,Fischer:18}, 
see details in Sections  \ref{sec:Matr_Solv} and \ref{sec:Comp_Ahom}.

In what follows, we present the numerical analysis for 2D stochastic 
homogenization problems (\ref{eqn:hm_setting})
with periodic boundary conditions on $\Gamma = \partial \Omega $, in the form
\[
 \phi(0,x_2)=\phi(L,x_2), \quad 
 \frac{\partial }{\partial x_1}\phi(0,x_2)=\frac{\partial }{\partial x_1}\phi(L,x_2),
 \quad x_2\in [0,L),
\]
\[
 \phi(x_1,0)=\phi(x_1,L), \quad 
 \frac{\partial }{\partial x_2}\phi(x_1,0)=\frac{\partial }{\partial x_2}\phi(x_1,L),
 \quad x_1\in [0,L).
\]
The diagonal $2\times 2$ coefficient matrix $\mathbb{A}(x)$ is defined by
 \[
  \mathbb{A}(x)=\begin{pmatrix}
   a(x) & 0 \\
   0 & a(x)\\
    \end{pmatrix}, \quad    \quad x\in\Omega,
 \]
where the scalar function $a(x)>0$ is piecewise constant in $\Omega$.
The efficient numerical simulation pre-supposes  the fast numerical solution of the equation (\ref{eqn:hm_setting})
in the case of many different coefficients $\mathbb{A}(x)$ and right-hand sides, 
generated by certain stochastic procedure,
and in the calculation of various functionals on the sequence of solutions $\phi$. 
In this problem setting the bottleneck task is fast and accurate generation of the FEM stiffness 
matrix in the sparse matrix form, which should be re-calculated many hundred if not thousand times
in the course of stochastic realizations.


In asymptotic analysis of stochastic homogenization problems the coefficient and
the right-hand side are chosen in a specific way,  see \cite{GlOtto:12,GlOtto:16} 
for the particular problem setting.
In this paper,  we describe the conventional 2D FEM discretization scheme 
in the rescaled domain  $\Omega=[0,1)^2$. 
Given the size of representative volume elements $L=1,2,3,\ldots$  
we generate (possibly overlapping) decomposition of the domain  $\Omega$ into 
$L^2$ equal unit cells $G_s$, $s=1,\ldots , L^2$, each of size 
$\frac{2\alpha}{L}\times \frac{2\alpha}{L}$,
whose centers are distributed randomly  (by the Poisson distribution) within the supercell $ \Omega$,
taking into account periodicity in both spacial variables $x_1$ and $x_2$.
Here the overlap factor satisfies $\alpha\leq 1/2$.
Stochastic characteristics of the system 
can be estimated  at the limit of a large number $L$. 

We consider a sequence of random coefficient distributions
$\{G_s\}_n$, numbered by $n=1,\ldots , N$, where the particular 
set $\{G_s\}=\{G_s\}_n$ for fixed $n$  will be called a realization. 
For any fixed realization define the covered domain
\begin{equation}\label{eqn:Stoch-decomp}
  \widehat{G}=\widehat{G}_n := \bigcup^{L^2}_{s=1} G_s,
\end{equation}
and the respective coefficient 
\begin{equation}\label{ahm_coef}
 \widehat{a}(x)=\widehat{a}^{(n)}(x) 
 = \left\{ \begin{array}{ll}
         1 & \mbox{if } x \in \widehat{G}_n ,  \\
        0 & \mbox{otherwise} .
        \end{array} \right. 
\end{equation}
The stochastic model is specified by the choice of the overlap constant $\alpha\in (0,1/2]$ and
the scaling factor $\lambda \in (0,1]$.
In the following, the constant $\lambda$ will be
 fixed in the interval $0.1 \leq \lambda \leq 0.8$. 
We denote the ``stochastic'' elliptic operator for the particular realization by
 ${\cal A}^{(n)}= {\cal A}_{\lambda}^{(n)}$ or just ${\cal A}$  (if $n$ is fixed) so that
 \[
  {\cal A}^{(n)} = - \nabla \cdot \mathbb{A}^{(n)}(x) \nabla,
 \]
where  the corresponding  $2\times 2$ coefficient matrix $\mathbb{A}^{(n)}(x)$ is defined by
 \begin{equation}\label{eqn:Am_coef}
  \mathbb{A}^{(n)}(x)= 
  \lambda  \begin{pmatrix}
   1 & 0 \\
   0 & 1\\
   \end{pmatrix} + 
  (1-\lambda)\begin{pmatrix}
   \widehat{a}^{(n)}(x) & 0 \\
   0 & \widehat{a}^{(n)}(x)\\
    \end{pmatrix} := 
     \begin{pmatrix}
   {a}^{(n)}(x) & 0 \\
   0 & {a}^{(n)}(x)\\
   \end{pmatrix},     \quad x\in\Omega,
 \end{equation}
 and the diagonal matrix coefficient takes the form 
 \begin{equation}\label{eqn:coef_diag} 
 {a}^{(n)}(x)=\lambda + (1-\lambda)\widehat{a}^{(n)}(x).
\end{equation}
We use the notation $\widehat{\mathbb{A}}^{(n)}(x)$ for the 
 "stochastic part" of a matrix associated with  the diagonal coefficient $\widehat{a}^{(n)}(x)$, i.e.
 \[
   \widehat{\mathbb{A}}^{(n)}(x)=
   \begin{pmatrix}
   \widehat{a}^{(n)}(x) & 0 \\
   0 & \widehat{a}^{(n)}(x)\\
    \end{pmatrix}.
 \]

Now the elliptic equations in stochastic homogenization are formulated as follows.
Fixed coefficient $\widehat{\mathbb{A}}^{(n)}(x)$, for $i=1,2$ solve the periodic 
 elliptic problems in $\Omega$,
 \begin{equation}\label{eqn:RHS_grad1}
 -\lambda \Delta\phi_i -(1-\lambda)\nabla \cdot \widehat{\mathbb{A}}^{(n)}
 ({\bf e}_i +\nabla \phi_i )=0,
 \end{equation}
 where the directional unit vectors ${\bf e}_i$, $i=1,2$, are given by ${\bf e}_1=(1,0)^T$ 
 and  ${\bf e}_2=(0,1)^T$, see Section \ref{sec:Comp_Ahom} for more details.
 The variational form of the equation (\ref{eqn:RHS_grad1}) reads as follows
 \begin{equation}\label{eqn:Ellipt_variat}
\int_{\Omega}(\lambda \nabla \phi_i  \cdot \nabla \psi + 
(1-\lambda) \widehat{a}^{(n)}(x)[{\bf e}_i + \nabla \phi_i]\cdot \nabla\psi )dx =0 
\quad \forall \; \psi \in H^1(\Omega).
\end{equation}
  The equation (\ref{eqn:RHS_grad1})  can be also written in the classical form (\ref{eqn:hm_setting})
 \begin{equation}\label{eqn:scaled_setting}
 {\cal A}^{(n)} \phi_i =  f_i, \quad \mbox{with} \quad f_i=
 (1-\lambda)\nabla \cdot \widehat{a}^{(n)}\, {\bf e}_i. 
 \end{equation}

 Fig. \ref{fig:Coeff} illustrates an example of the particular realization of stochastic coefficient $a^{(n)}(x)$ 
in the case $L=8$, $\lambda=0.1$ and $\alpha =1/4$ visualized on $m_1\times m_1$ grid with $m_1=97$.
 \begin{figure}[htbp]
\centering
\includegraphics[width=7.6cm]{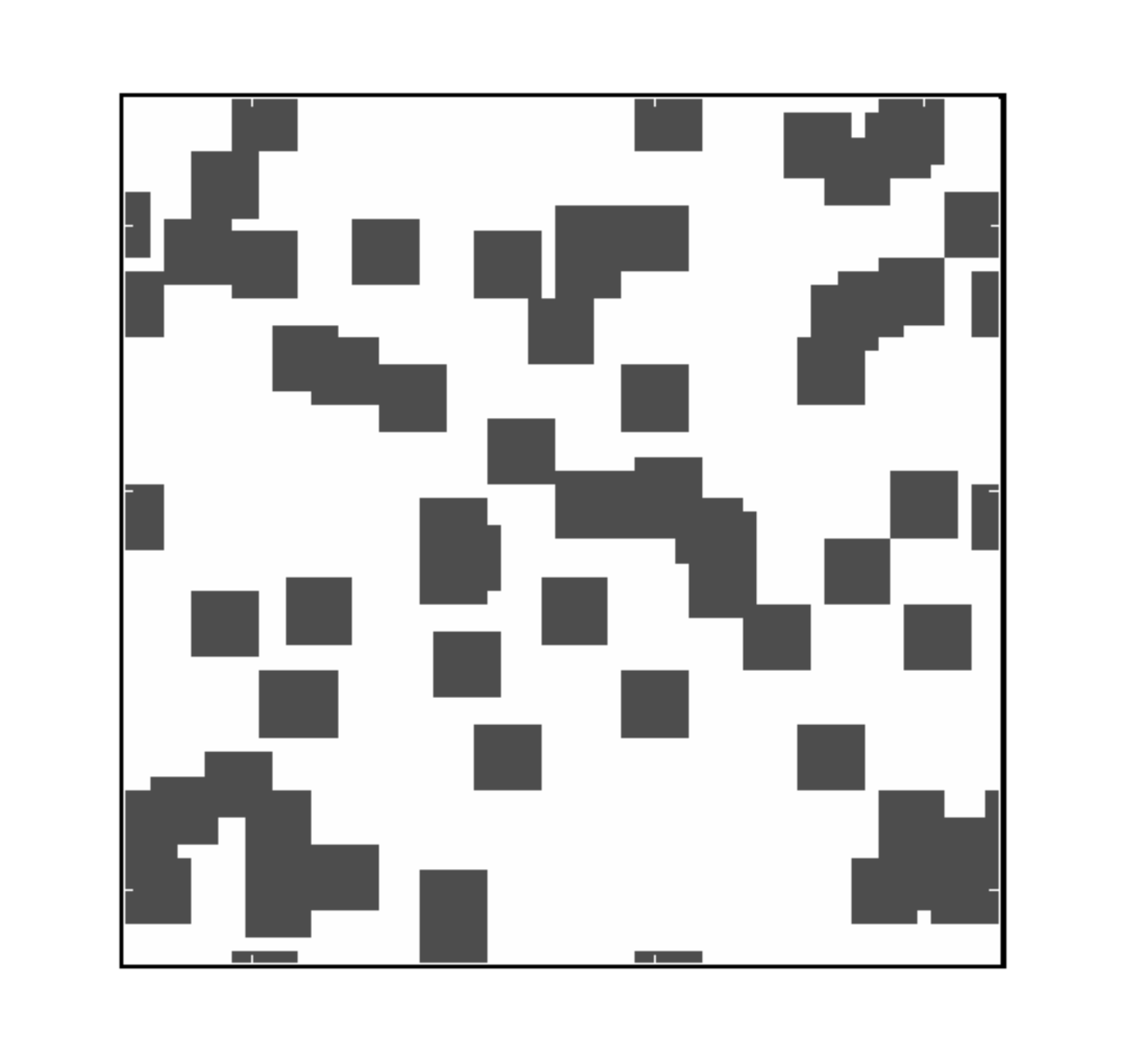}
\includegraphics[width=8.4cm]{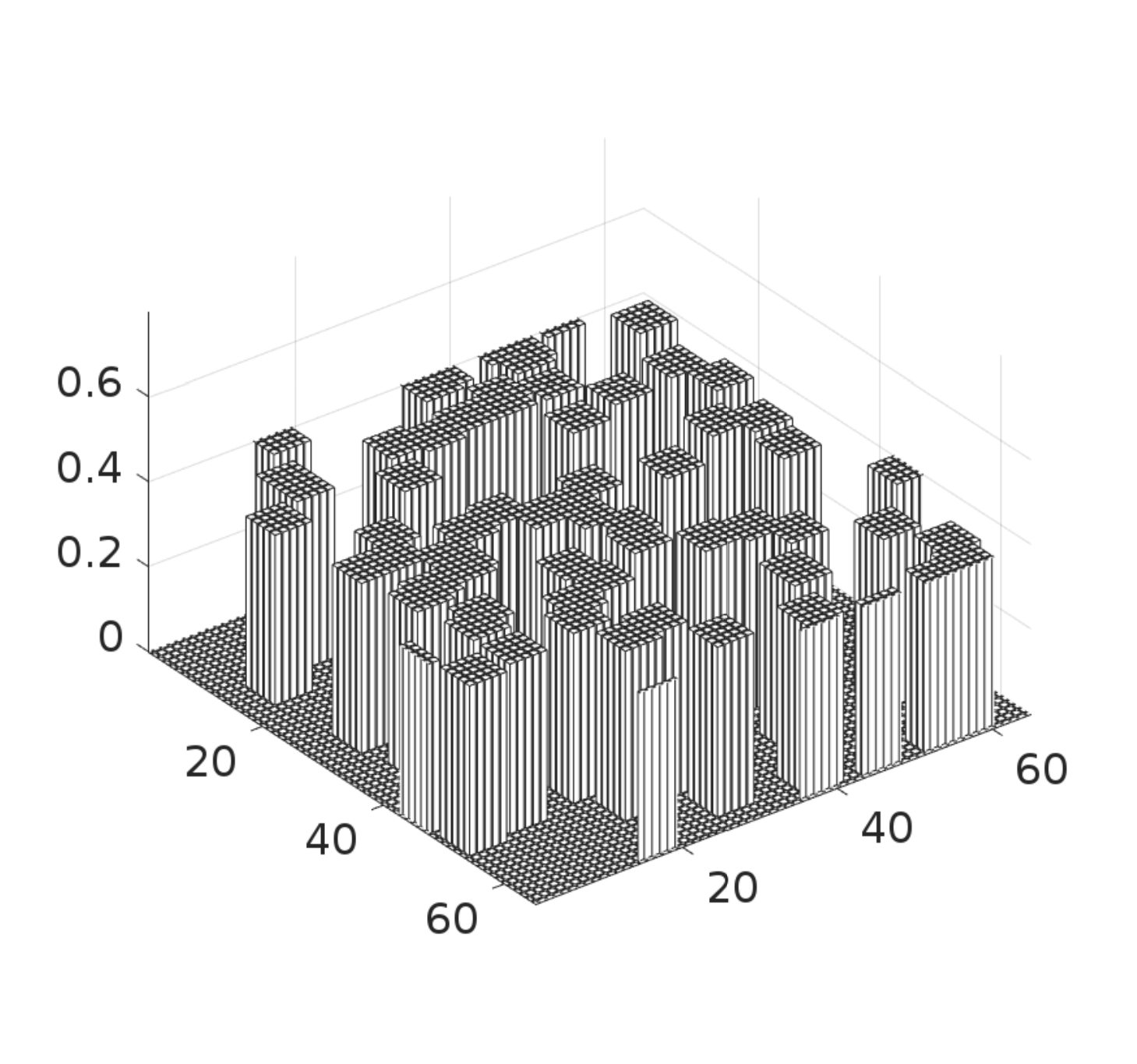} 
\caption{\small A realization of a stochastic process with $L^2$ overlapping cells
for $L=8$, $\alpha =1/4$.}
\label{fig:Coeff}
\end{figure}

The problem setting remains verbatim in the $d$-dimensional case, $d>2$. In this case the 
equation (\ref{eqn:scaled_setting}) takes the same form, where a $d\times d$ coefficient matrix 
is given by
\[
 \mathbb{A}^{(n)}(x)= \mbox{diag}\{a^{(n)}(x),\ldots,a^{(n)}(x) \}, \quad x \in \Omega
\]
and  ${\bf e}_i\in \mathbb{R}^d$, $i=1,\ldots, d$, represent the set of directional 
unit vectors in $\mathbb{R}^d$.

\section{Matrix generation and iterative solution}\label{sec:Matr_Solv} 

In this section we describe the FEM discretization scheme and the fast matrix generation 
approach based on the use of tensor Kronecker products of ``univariate''   matrices. 

\subsection{Galerkin FEM discretization}\label{ssec:FEM_FDM}

 First, we introduce the uniform $m_s\times m_s$ rectangular grid $\Omega_{h_s}$ in $\Omega$
 with the grid size $h_s=\frac{1}{m_s -1}$, such that $m_s = m_0 L +1$, $m_0= 2^{p_0}$, i.e.
 $h_s = \frac{1}{m_0 L}$. We assume that the unit cell $G_s$, $s=1,\ldots L^2$, of size 
 $\frac{2\alpha}{L} \times \frac{2\alpha}{L}$ adjust the square grid $\Omega_{h_s}$, such that
 the center $c_s$ of $G_s$ belongs to the set of grid points in $\Omega_s$,
 while the overlap factor $\alpha $ may take values 
 $\alpha \in \{\frac{1}{m_0}, \frac{2}{m_0},\ldots \frac{2^{p_0 -1}}{m_0}\}$. 
 In this construction the univariate size of the unit cell varies as
 \[
  \frac{2\alpha}{L} = \frac{2\alpha m_0}{m_0 L} = k h_s, \quad \mbox{with } \quad
  k=2,\; 4,\; \ldots m_0.
 \]
 In the following numerical examples we normally use the overlap constant $\alpha =1/4$.
 For $\alpha =1/2$, the maximal size of the unit cell is given by $\frac{1}{L}\times \frac{1}{L}$, which  
contains $m_0+1$ grid points in each spacial direction leading to $m_1\times m_1$ rectangular
grid with $m_1=m_0 L +1$.

The FEM discretization of the elliptic PDE in (\ref{eqn:scaled_setting}) can be constructed, 
in general, on the finer grid $\Omega_h$ compared with $\Omega_s$, which serves for the resolution of 
jumping coefficients. 
To that end, we introduce the $m_1\times m_1$ rectangular
grid $\Omega_h$ with the mesh size $h=\frac{1}{m_1-1} $, $m_1 \geq m_s$, 
that is obtained by a dyadic refinement
of the grid $\Omega_s$, such that the relation 
\begin{equation}\label{eqn:Grid_relation}
 m_1-1=(m_s -1) 2^{p}, \quad \mbox{with} \quad p=0,1,2,\ldots
\end{equation}
holds, implying $h_s = 2^p h$. Now the grid-size of the unit cell $G_s$ 
on the finer grid $\Omega_h$ is given by $(m_0 2^p +1)\times (m_0 2^p +1)$.

Given a finite dimensional space $X\subset H^1(\Omega)$ of tensor product piecewise linear finite elements 
$X = \mbox{span}\{ \psi_\mu(x) \}$ associated with the grid $\Omega_h$, with  $\mu=1,...,M_d $, $M_d=m_1^d$,
for $d=2$ incorporating periodic boundary conditions,  we are looking for the traditional 
FEM Galerkin approximation of the exact solution in the form 
$$
\phi(x) \approx \phi_X(x)=\sum_{\mu=1}^{M_d} u_\mu \psi_\mu(x) \in X,
$$
where ${\bf u}=(u_1,\ldots,u_{M_d})^T\in  \mathbb{R}^{M_d}$ is the unknown coefficients vector.
Fixed realization of the coefficient $a^{(n)}(x)$, for $i=1,2$
we define the Galerkin-FEM discretization with respect to $X$ of the variational equation 
(\ref{eqn:Ellipt_variat})  by 
\begin{equation}\label{eqn:FEM_Galerk}
 A {\bf u}_i = {\bf f}_i, \quad A = [a_{\mu \nu}]\in \mathbb{R}^{M_d\times M_d},\quad 
 {\bf f}_i = {\bf f}=[f_\mu] \in  \mathbb{R}^{M_d}, 
\end{equation}
where the Galerkin-FEM matrix $A$ generated by the equation coefficient 
 ${\mathbb{A}}^{(n)}(x)$ is calculated by using the associated bilinear form
\begin{equation}\label{eqn:FEM_discr}
a_{\mu \nu} = \langle {\cal A} \psi_\mu,  \psi_\nu \rangle = 
\int_{\Omega}(\lambda \nabla \psi_\mu  \cdot \nabla \psi_\nu + 
(1-\lambda) a^{(n)}(x) \nabla \psi_\mu \cdot \nabla \psi_\nu)dx, 
\end{equation}
and 
\begin{equation}\label{eqn:FEM_discr_RHS}
 f_\mu = \langle f, \psi_\mu \rangle = 
 \int_\Omega (1-\lambda)\nabla \cdot \widehat{a}^{(n)}(x)\, {\bf e}_i \psi_\mu \, dx =
 - (1-\lambda) \int_\Omega \widehat{a}^{(n)}(x) \frac{\partial \psi_\mu}{\partial x_i} dx.
\end{equation}
Corresponding to (\ref{eqn:scaled_setting}) and (\ref{eqn:FEM_discr}), 
 we represent the stiffness matrix $A$ in the additive form
\begin{equation}\label{eqn:FEM_matrix}
A=\lambda A_\Delta + (1-\lambda) \widehat{A}_s,
\end{equation}
where $A_\Delta$ represents the $M_d \times M_d$
FEM Laplacian matrix in periodic setting that has the standard two-terms Kronecker product form. 
Here matrix $\widehat{A}_s$ provides the FEM approximation to the "stochastic part" 
in the elliptic operator corresponding to the coefficient $\widehat{a}^{(n)}(x)$, 
see (\ref{eqn:Am_coef}). The latter is determined by the sequence of
random coefficients distribution in the course of stochastic realizations, 
numbered by $n=1,\ldots,N$.

In the case of complicated jumping coefficients the stiffness matrix generation in the elliptic FEM 
usually constitutes the dominating part of the overall solution cost.
In the course of stochastic realizations the equation (\ref{eqn:FEM_Galerk})
is to be solved many hundred or even thousand times, so that every time 
one has to update the stiffness matrix $A$ and the right-hand side $\bf f$.

Our discretization scheme  computes  all matrix entries at the low cost
by assembling of the local Kronecker product matrices obtained by representation of
$\widehat{a}^{(n)}(x)$ as a sum of separable functions.
This allows to store the resultant stiffness matrix in the sparse matrix format. 
Such a construction only includes the 
pre-computing of tri-diagonal matrices representing 1D elliptic operators 
with jumping coefficients in periodic setting.
In the next sections, we shall describe the efficient construction 
of the "stochastic" term $A_s$.

\subsection{Matrix generation by using Kronecker product sums}\label{ssec:FDM_Kron} 


To enhance the time consuming matrix assembling process 
we apply the FEM Galerkin discretization (\ref{eqn:FEM_discr}) of 
equation (\ref{eqn:scaled_setting}) by means of the 
tensor-product piecewise linear finite elements 
$$
\{\psi_{\boldsymbol{\mu}}(x):=\psi_{\mu_1}(x_1) \cdots \psi_{\mu_d}(x_d)\}, 
\quad {\boldsymbol{\mu}}=(\mu_1,\ldots,\mu_d), 
\quad \mu_\ell\in {\mathcal I}_\ell=\{1,\ldots,m_\ell\}, \; \ell=1,\ldots,d,
$$
where $\psi_{\mu_\ell}(x_\ell)$ are the univariate piecewise linear hat 
functions\footnote{Notice that the univariate grid size $m_\ell$ is of the order of $m_\ell=O(1/\epsilon)$,
where the small homogenization parameter is given by $\epsilon\approx 1/(m_0 L)$,
designating the total problem size 
$
M_d = m_1m_2\cdots m_d=O(1/\epsilon^d).
$
}.
The $M_d\times M_d$ stiffness matrix is constructed by the standard mapping of the multi-index 
$ \boldsymbol{\mu}$
into the long univariate index $1\leq \mu \leq M_d$ for the active degrees of freedom in periodic setting. 
For instance,  we use the so-called big-endian convention for $d=3$ and $d=2$
\[
 {\boldsymbol{\mu}}\mapsto \mu:= \mu_3 + (\mu_2-1)m_3 + (\mu_1-1)m_2 m_3, 
 \quad {\boldsymbol{\mu}}\mapsto \mu:= \mu_2 + (\mu_1-1)m_2,
\]
respectively. 
In what follows,  we consider the case $d=2$ in more detail. 

In our discretization scheme  we calculate the stiffness matrix  
by assembling of the local Kronecker product terms by using representation of the ``stochastic part''
in the coefficient $\widehat{a}^{(n)}(x)$ as an $R$-term sum of separable functions.
To that end, let us assume for the moment that the scalar diffusion coefficient $a(x_1,x_2) $ can be 
represented in the separate form (rank-$1$ representation)
\[
 a(x_1,x_2) = a^{(1)} (x_1) a^{(2)} (x_2).
\]
Then the entries of the Galerkin 
stiffness matrix $A=[a_{{\mu}{\nu}}]\in \mathbb{R}^{M_d\times M_d}$ can be represented by
\begin{align*}\label{eqn:tensStMatr}
a_{{\mu}{\nu}}   &=   \langle {\mathcal A} \psi_{\mu}, \psi_{\nu} \rangle =
\int_\Omega  a^{(1)} (x_1) a^{(2)} (x_2) 
\nabla \psi_\mu \cdot \nabla \psi_\nu  d x \\
    & = {\int_{(0,1)} } a^{(1)} (x_1) 
    \dfrac{\partial \psi_{\mu_1} (x_1)}{\partial x_1 }   
   \dfrac{\partial \psi_{\nu_1} (x_1)}{\partial x_1 } d x_1   
   \int_{(0,1)} a^{(2)} (x_2) \psi_{\mu_2}(x_2)\psi_{\nu_2}(x_2) d x_2  \\
   & +  \int_{(0,1)} a^{(1)} (x_1) \psi_{\mu_1}(x_1) \psi_{\nu_1} (x_1)d x_1  
   \int_{(0,1)} a^{(2)} (x_2) \dfrac{\partial \psi_{\mu_2} (x_2)}{\partial x_2 } 
   \dfrac{\partial \psi_{\nu_2} (x_2)}{\partial x_2 } d x_2 , 
  \end{align*} 
which leads to the rank-2 Kronecker product representation 
\[
 {A} = {A}_1 \otimes S_2 + S_1 \otimes {A}_2,
\]
where $\otimes$ denotes the conventional Kronecker product of matrices, 
see Definition \ref{def:Kron} below.
Here ${A}_1=[a_{\mu_1\nu_1}]\in \mathbb{R}^{m_1\times m_1} $ and 
${A}_2=[a_{\mu_2 \nu_2}]\in \mathbb{R}^{m_2\times m_2}$ denote the univariate stiffness matrices
and $S_1=[s_{\mu_1\nu_1}]\in \mathbb{R}^{m_1\times m_1}$ and 
$S_2=[s_{\mu_2 \nu_2}]\in \mathbb{R}^{m_2\times m_2}$ define the weighted mass matrices, 
for example
\[
 a_{\mu_1\nu_1}= 
 {\int_{(0,1)} } a^{(1)} (x_1) 
    \frac{\partial \psi_{\mu_1} (x_1)}{\partial x_1 }   
   \frac{\partial \psi_{\nu_1} (x_1)}{\partial x_1 } d x_1 , \quad
  s_{\mu_1\nu_1}= 
  \int_{(0,1)} a^{(1)} (x_1) \psi_{\mu_1}(x_1) \psi_{\nu_1} (x_1)d x_1 .
\]
\begin{definition} \label{def:Kron}
Recall that given $p_1\times q_1$ matrix $A$ and 
$p_2\times q_2$ matrix $B$, their Kronecker product is defined as a $p_1 p_2 \times q_1 q_2$ 
matrix $C$ via the block representation 
$$ 
C=A \otimes B =[a_{ij}B], \quad i=1,\ldots,p_1, \; j=1,\ldots,q_1.
$$
\end{definition}

Let us discuss in more detail the calculation of the 1D stiffness matrices $A_1$ and $A_2$
in the case of variable 1D coefficients.
We choose  the Galerkin FEM with $m=m_1$ piecewise-linear hat 
functions $\left\{\psi_{\mu_1} \right\}$ in periodic setting in $\Omega =[0,1)$, 
constructed on a uniform grid with a step size $h=1/m$, 
and nodes $x_{\mu_1}=h \,{\mu_1}$, ${\mu_1} = 1,\ldots,m$.  
If we denote the diffusion coefficient by $a(x_1)$, then
the entries of the exact stiffness matrix ${A}_1$  read as
\begin{equation*}\label{eqn:GalekMatr}
\left(a \right)_{\mu_1,\mu_1'} = \langle a(x)\nabla\psi_{\mu_1}(x), 
\nabla\psi_{\mu_1'}(x) \rangle_{L_2(D)}, \quad {\mu_1},{\mu_1'}=1,\ldots,m.
\end{equation*}
We assume that the coefficient remains constant at each
spatial interval $[x_{\mu_1-1}, x_{\mu_1}]$, which corresponds to the evaluation of the scalar product
above via the midpoint quadrature rule yielding the approximation order $O({h^2})$.

Introducing the coefficient vector
${\bf a}=[a_{\mu_1}]\in \mathbb{R}^m$, $a_{\mu_1} = a(x_{\mu_1-1/2})$, $\mu_1=1,\ldots,m$, 
where $x_{i_1-1/2}$ is the middle point of the integration interval,
the symmetric tridiagonal matrix of interest can be represented by
\begin{equation}
{A}_1 = \dfrac{1}{h}\left[
\begin{matrix}
a_1+a_2 & -a_2 & & & -a_1\\
-a_2 & a_2+a_3 & -a_3 \\
& \ddots & \ddots & \ddots \\
& & -a_{m-1} & a_{m-1}+a_{m} & -a_{m} \\
-a_1& & & -a_{m} & a_{m} +a_1
\end{matrix}\right].
\label{eqn:stiff_matrix_1point}
\end{equation}

By simple algebraic transformations (e.g. by lamping of the mass matrices)
the matrix ${A}$ can be simplified to the form (without loss of approximation order)  
\begin{equation} \label{eqn:Lapl_Kron_D}
 {A} \mapsto A = A_1 \otimes D_2 + D_1 \otimes A_2,
\end{equation}
 where $D_1, D_2$ are the diagonal matrices with positive entries.
 This representation applies in particular to the periodic Laplacian.
 
 In the general case, the piecewise constant stochastic coefficient can be represented 
 as an $R$-term sum of separable coefficients.  This leads to the linear system of equations
\begin{equation} \label{eqn:FEM_syst}
 A {\bf u} = {\bf f},
\end{equation}
constructed for the general $R$-term separable coefficient $a(x_1,x_2)$ with $R>1$.

The representation in (\ref{eqn:Lapl_Kron_D}) can be further simplified to the 
anisotropic Laplacian type matrix
\begin{equation*} \label{eqn:Lapl_Kron}
 {A} \mapsto B = \alpha_2 A_1 \otimes I_2 + \alpha_1 I_1 \otimes A_2,
\end{equation*}
which will be used as a prototype preconditioner for solving
the target linear system (\ref{eqn:FEM_syst}).

Taking into account the rectangular structure of the grid, 
we use the simple finite-difference (FD) scheme for the matrix
representation of the Laplacian operator $-\Delta$.
In this case the scaled discrete Laplacian incorporating periodic boundary conditions takes the form
\begin{equation}\label{eqn:Lap_Kron}
 A_\Delta  = \Delta_1 \otimes I_{m_2} + I_{m_1}\otimes \Delta_2,
\end{equation}
where
\[
 -\Delta_1 = \mathrm{tridiag} \{ 1,-2,1 \} + P^{(1)} \in \mathbb{R}^{m_1\times m_1},
\] 
such that the entries of the "periodization" matrix $P^{(1)}\in \mathbb{R}^{m_1\times m_1}$ 
are all zeros except 
$$ 
P^{(1)}_{1,m_1}=P^{(1)}_{m_1,1}=1, \quad \mbox{and} \quad  P^{(1)}_{1,1}=P^{(1)}_{m_1,m_1}=-1.
$$ 
Here
$I_{m_1}\in \mathbb{R}^{m_1 \times m_1}$ is the identity matrix, $\Delta_1 =\Delta_2$
is the 1D finite difference Laplacian (endorsed with the Neumann boundary conditions),
and $\otimes$ denotes the Kronecker product of matrices, see Definition \ref{def:Kron}.
We say that the Kronecker rank of both $A$ in (\ref{eqn:Lapl_Kron_D}) and 
$A_\Delta$ in (\ref{eqn:Lap_Kron}) equals to $2$.

Notice that the $m_1 \times m_1$ Laplacian matrices for the Neumann and periodic boundary 
conditions in 1D read as 
 \begin{equation}\label{BC_Neu_peri}
 \Delta_{N} = 
 \begin{bmatrix}
  -1 &  1 &  \cdots &  0 &  0\\
   1 & -2 & \cdots &  0 &  0 \\
   \vdots & \vdots & \ddots  &  \vdots & \vdots \\ 
   0 &  0 &  \cdots & -2 &  1 \\ 
   0 &  0 &  \cdots &  1 & -1 \\ 
 \end{bmatrix}
 \quad \mbox{ and }\quad
 \Delta_P=
  \begin{bmatrix}
  -2 &  1 &  \cdots &  0 &  1\\
   1 & -2 & \cdots &  0 &  0 \\
   \vdots & \vdots & \ddots  &  \vdots & \vdots \\ 
   0 &  0 &  \cdots & -2 &  1 \\ 
   1 &  0 &  \cdots &  1 & -2 \\ 
 \end{bmatrix},
  \end{equation}
respectively. 
 
In the $d$-dimensional case we have the similar Kronecker rank-$d$ representations.
For example, in the case  $d=3$ the "periodic" Laplacian $M_d\times M_d$ matrix 
$A_{\Delta}$ takes a form
\[ 
A_{\Delta} = A_{1,P} \otimes I_2\otimes I_3 + I_1 \otimes A_{2,P} \otimes I_3 + 
I_1 \otimes I_2\otimes A_{3,P},
\] 
such that its Kronecker rank equals to $3$, and similar for the arbitrary $d\geq 3$.

 \subsection{Fast matrix assembling for the stochastic part}\label{ssec:FDM_Kron_stoch} 
 
The Kronecker form representation of the "stochastic" term in (\ref{eqn:FEM_discr})
 further denoted by $A_s$ is more involved. For given stochastically chosen distribution 
 of overlapping cells $G_s$, $s=1,\ldots,L^2$, we construct the minimal non-overlapping 
 decomposition of the full covered  grid domain 
 $\widehat{G} = \cup^{L^2}_{s=1} G_s$ colored by gray in Figure \ref{fig:Coeff}  
 (we have $a(x)=1$ for $x\in \widehat{G}$ and $a(x)=\lambda$ for 
 $x\in \Omega  \setminus \widehat{G}$)  in a form of a union
 of elementary square cells $S_k$, $k=1,\ldots,K$, $K\geq L^2$, each of the grid-size  
 $\overline{m}_0\times \overline{m}_0$,
 \begin{equation}\label{eqn:NO-decomp}
  \widehat{G}= \cup^{K}_{k=1} S_k.
 \end{equation}
 Here $\overline{m}_0= 2^p+1$, and
  $p=0,1,2,\ldots$, is fixed as above by relation $m_1 -1 =(m_s -1) 2^p$.
 In this construction, the non-overlapping elementary cells $S_k$ for different $k$ are allowed to
 have the only common edges of size $\overline{m}_0$.  
 Notice that in the case of non-overlapping
 decomposition (\ref{eqn:Stoch-decomp}) the set of cells $\{S_k\}$ coincides with  
 the initial set $\{G_s\}$ which allows to maximize the  size $\overline{m}_0\times \overline{m}_0$ 
 of each $S_k$, $k=1,\ldots,L^2$, 
 to the largest possible, i.e. to $\overline{m}_0= m_0 2^p+1$. 
  
 To finalize the matrix generation procedure for $A_s$, we define the local 
 $\overline{m}_0\times \overline{m}_0$
 matrices representing the discrete Laplacian with Neumann boundary conditions,
 \[
  \widehat{Q}_{\overline{m}_0}:= \mbox{tridiag} \{ 1,-2,1 \} + \mbox{diag}\{ 1,0,\ldots,0,1 \}
  \in \mathbb{R}^{\overline{m}_0\times \overline{m}_0},
 \]
and the diagonal matrix 
$$
\widehat{I}_{\overline{m}_0}:=\mbox{diag}\{ 1/2,1,\ldots,1,1/2 \}
\in \mathbb{R}^{\overline{m}_0\times \overline{m}_0},
$$
see the visualization in (\ref{BC_Neu_peri}). 
Here, we may select $\overline{m}_0=2,3,5,...$ that corresponds to the choice $p= 0,1,2,...$.
In the case of $\overline{m}_0 \times \overline{m}_0 $ matrix with minimal size 
$\overline{m}_0=2$,  both discrete Laplacians in (\ref{BC_Neu_peri}) simplify to
\begin{equation}\label{matr_mini}
   \Delta_N=   
   \begin{bmatrix}
  -1 &  1 \\
   1 &  -1\\ 
 \end{bmatrix}
 \quad \mbox{ and }\quad
      \Delta_P= 
      \begin{bmatrix}
  -1 &  1 \\
   1 &  -1\\ 
 \end{bmatrix}.
 \end{equation}
 
 Let the subdomain $S_k$ be supported by the index set $I_k^{(1)}\times I_k^{(2)}$ 
of size $\overline{m}_0\times \overline{m}_0$ for $k=1,\ldots,K$.
Introduce the block-diagonal matrices $\overline{Q}_k \in \mathbb{R}^{m_1\times m_1}$
and $\overline{I}_k \in \mathbb{R}^{m_1\times m_1}$ by inserting matrices 
$\widehat{Q}_{\overline{m}_0}$ 
and $\widehat{I}_{\overline{m}_0}$ as diagonal blocks into $m_1\times m_1$ zero matrix in the positions 
$I_k^{(1)}\times I_k^{(1)}$ and $I_k^{(2)}\times I_k^{(2)}$, respectively.

Now the stiffness matrix $A_s$ is represented in the form of a Kronecker product sum as follows,
\begin{equation}\label{eqn:As_Kron}
 A_s= \sum^{K}_{k=1} (\overline{Q}_k \otimes \overline{I}_k + 
 \overline{I}_k \otimes \overline{Q}_k) + P^{(2)},
\end{equation}
where 
$$
P^{(2)}= P^{(1)} \otimes I_{m_1} + I_{m_1} \otimes P^{(1)} \in \mathbb{R}^{M_d\times M_d}
$$ 
is the "periodization" matrix in 2D.
In a $d$-dimensional case the representation (\ref{eqn:As_Kron}) generalizes 
to a sum of $d$-factor Kronecker products
\begin{equation}\label{eqn:As_Kron_d}
 A_s= \sum^{K}_{k=1} (\overline{Q}_k \otimes \overline{I}_k\otimes\cdots\otimes \overline{I}_k+
\ldots + \overline{I}_k \otimes \cdots \otimes \overline{I}_k\otimes \overline{Q}_k)
 + P^{(d)},
\end{equation}
where $P^{(d)}$ is the "periodization" matrix in $d$ dimensions, constructed 
as the $d$-term Kronecker sum similar to the case $d=2$.

The Kronecker product form of (\ref{eqn:Lap_Kron}) and (\ref{eqn:As_Kron}) leads  to
the corresponding Kronecker sum representation for the total stiffness matrix $A$. 
This allows an efficient implementation of 
the matrix assembly and low storage for the stiffness matrix preserving the Kronecker sparsity. 
Hence, it proves the following storage complexity for the matrix $A$.
\begin{lemma}\label{lem:Matr_stor}
 The storage size for the stiffness matrix $A$  is bounded by
 \[
 Stor(A)\approx  Stor(A_s) = O(d \overline{m}_0 K + d m_1).
 \]
\end{lemma}

Here, in general, the number $K$ of elementary cells\footnote{For example, 
for cells of minimal size, $\overline{m}_0  \times \overline{m}_0$ 
with $\overline{m}_0=2$, as in (\ref{matr_mini}), we have $K=O(m_1^2)$.} 
is larger than $L^2$, and it coincides with $L^2$
 only in the case of non-overlapping decomposition $\widehat{G} = \cup^{L^2}_{s=1} G_s$,
where different patches $G_s$ are allowed to have joint pieces of boundary but no overlapping area.

In the general case $d\geq 2$ and $K\geq L^d$, the Kronecker rank of the 
matrix ${A}$ is bounded by
\[
 \mbox{rank}_{Kron} ({A}) \leq  d \, K.
\]
The Kronecker rank of the stiffness matrix reduces dramatically in two cases:\\
(a) For the case of non-overlapping cells $G_s$, $s=1,\ldots,L^2$, we have 
\[
\mbox{rank}_{Kron} ({A}) \leq L^d.
\]
(b) In the case of cell-centered locations of subdomains $G_s$ 
(special case of geometric homogenization) there holds 
\[
\mbox{rank}_{Kron} ({A}) \leq L^{d-1}. 
\]
 
\begin{figure}[htbp]
\centering
\includegraphics[width=4.0cm]{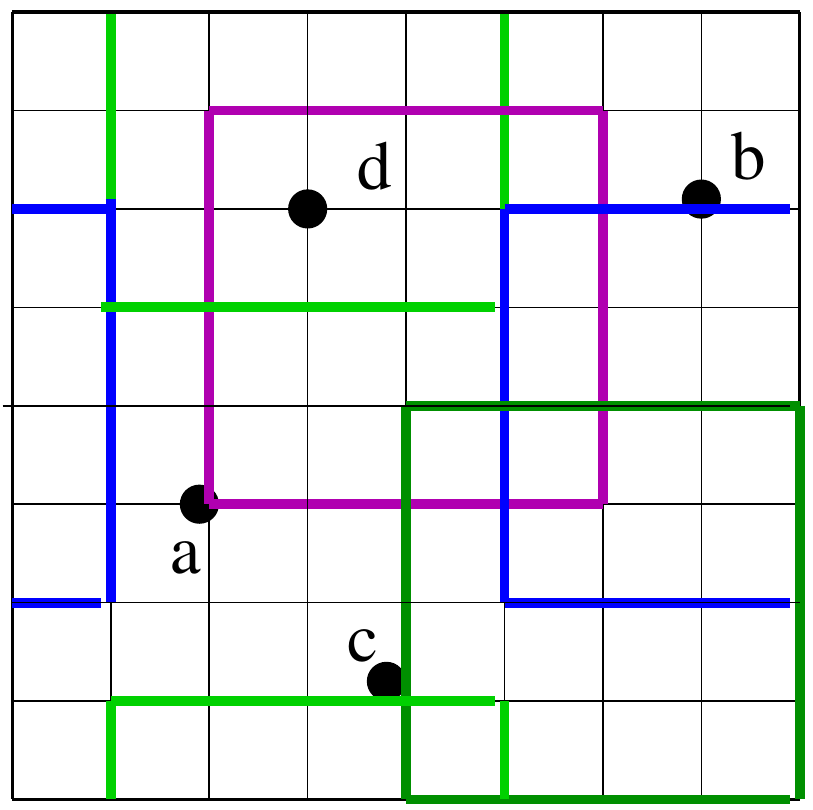}\qquad \quad \quad 
\includegraphics[width=4.0cm]{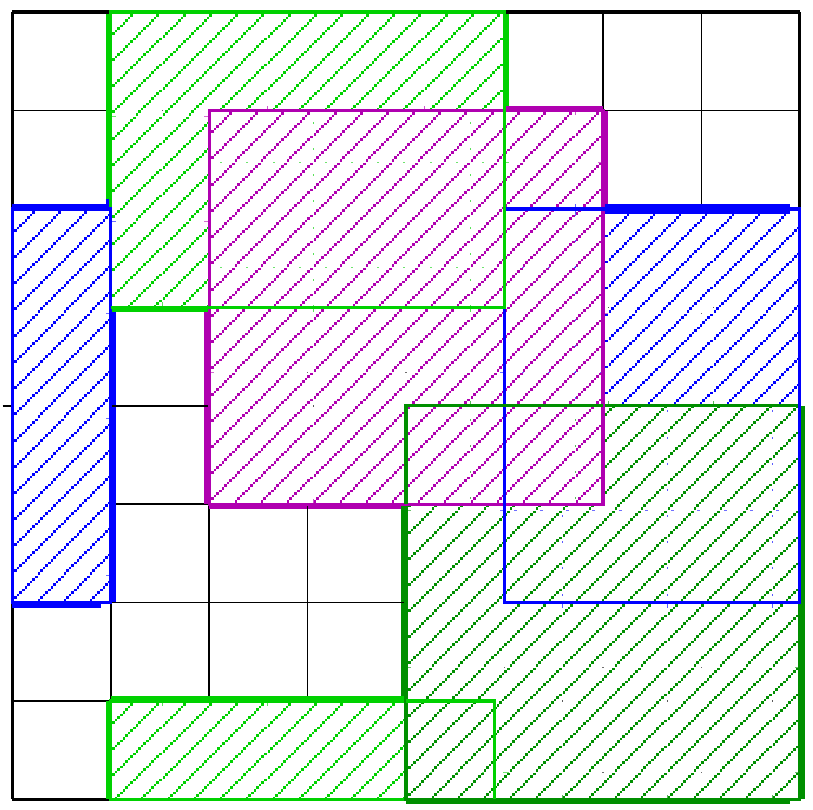} 
\caption{\small Example of the covering domain $\widehat{G}$ (right) and the typical locations of
sampling points for the grid representation of $\mathbb{A}_h(x_h)$ (left).}
\label{fig:CovDomain}
\end{figure} 
The corresponding vector representation ${\bf f}_i\in \mathbb{R}^{M_d}$ of the right hand side
$f_i(x)$ is computed by multiplication
of the discrete upwind gradient matrix $\nabla_h$ with a vector ${\bf y}_i\in \mathbb{R}^{M_d}$. 
Here the vector ${\bf y}_i$ represents 
the multiple of  the vector ${\bf e}_i$, $i=1,2$, and the equation 
coefficient $\mathbb{A}^{(n)}=\mathbb{A}(x) =\mbox{diag}\{{a}^{(n)}(x),{a}^{(n)}(x)\}$,  
discretized on 
the grid $\Omega_h$, i.e. each block-entry of the "discretized" matrix coefficient 
$\mathbb{A}(x)\mapsto A_h(x_h)$ is given by an $M_d$-vector array with $M_d=m^2$,
\[
A_h(x_h) = \mbox{diag}\{{a}^{(n)}(x_h),{a}^{(n)}(x_h)\}, \quad 
  {a}^{(n)}(x_h)_{|x_h\in \Omega_h} \in \mathbb{R}^{M_d}.
\]
Hence, we finally arrive at
\[
 {\bf f}_i= (1-\lambda)\nabla_h\cdot {\bf y}_i, \quad {\bf y}_i=
 [{\bf y}_i(x_h)]\in \mathbb{R}^{M_d}
 \;\; \mbox{with} \;\;{\bf y}_i(x_h)={A}_h(x_h){\bf e}_i, \quad x_h\in \Omega_h,
\]
for $i=1,2$.
Specifically, given the grid-point $x_h\in \Omega_h$, the corresponding diagonal value of 
${A}_h(x_h)$ is defined by $a^{(n)}(x_h)$, see (\ref{eqn:coef_diag}).
Here the variable part $\widehat{a}^{(n)}(x_h)$, describing the jumping coefficient, is assigned
by $1$ for interior points in $\widehat{G}$,  
by $1/2$ for interface points (the angle equals to $\pi/2$), by $3/4$ 
for the "interior" L-shaped corners 
(the angle equals to $3\pi/4$) and by $1/4$ for the "exterior" corner of $\widehat{G}$ 
(the angle equals to $\pi/4$), see points (d), (b), (c) and (a) in 
Figure \ref{fig:CovDomain}, respectively.
This figure corresponds to $L=2$, the discretization parameter $n_0=4$ and periodic completion 
of the geometry. One observes the complicated shape of the strongly jumping coefficients.

\subsection{Numerical analysis of the FEM approximation error }\label{ssec:DiscErr}

We tested convergence of the solutions  
on a sequence of dyadic refined grids, for the fixed configuration of coefficients and the 
right hand side given by $f=sin(2\pi x) cos(6\pi y)$.
Test examples are performed for $L=2, \; 4,\; 8 $, corresponding to $4$, $16$ and $64$  
bumps in the  coefficients, respectively.   
We compare the solution vectors ${\bf u}_p$ calculated on a sequence of five dyadic refined grids with 
the grid size $M_d=m_p^2=M_{d,p}$, with $m_p=2^{4+p}-1$, $p=1,\ldots,5$, equal to
$M_{d,p}=31^2$, $63^2$, $127^2$, $255^2$, $511^2$ and $1023^2$, respectively.
The matrix size is given by $ M_d\times M_d$. A FEM interpolation error in the $H^r(\Omega)$-norm
is expected of the order of $O(m^{-\beta +r})$ for $\beta \in (0,2]$ and $r\in [0,1]$, where 
$h=O(1/m)$ and $\beta$ measures the regularity of the solution  $u\in H^\beta(\Omega)$.
\begin{figure}[htbp]
\centering
\includegraphics[width=7.6cm]{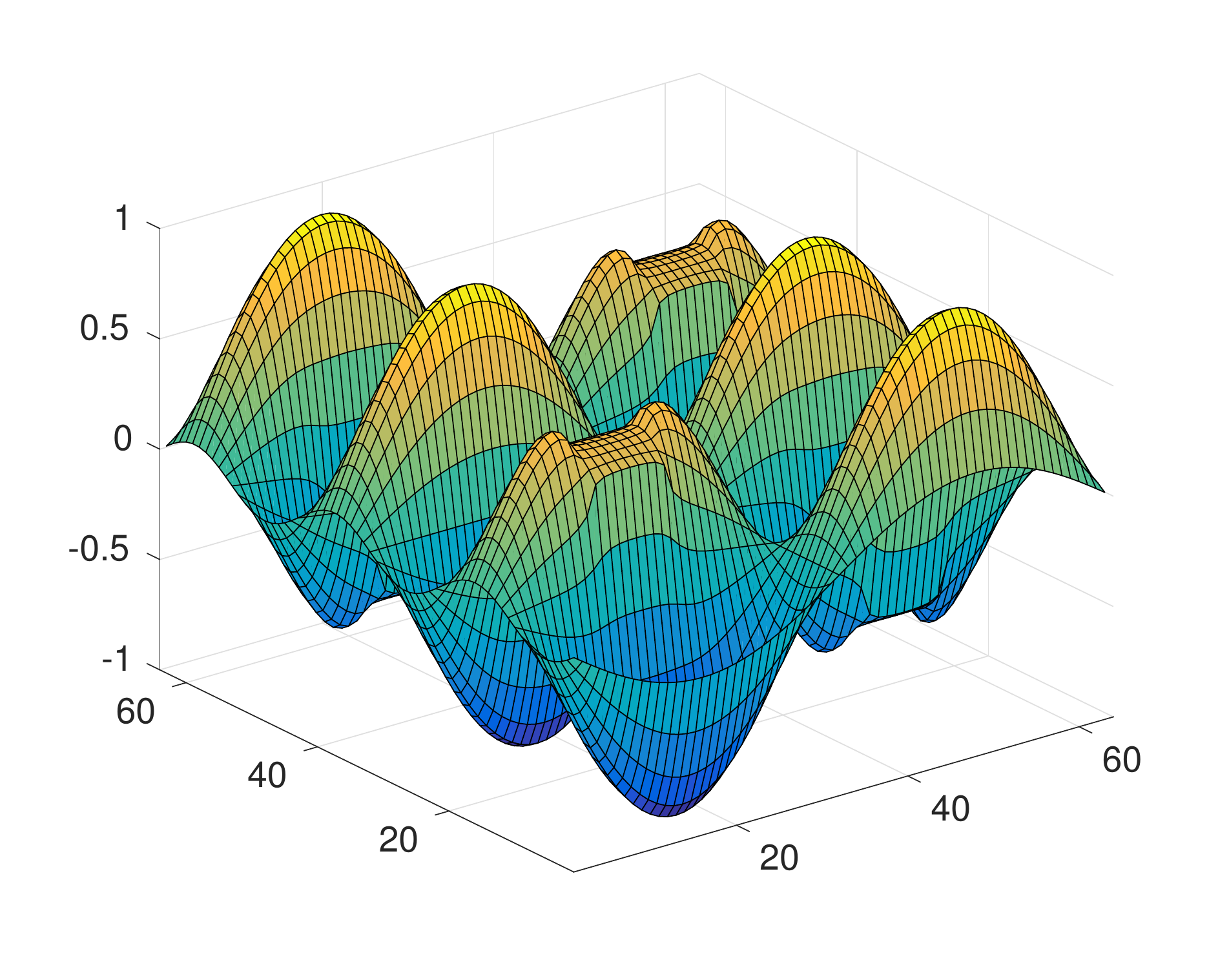}
\includegraphics[width=7.6cm]{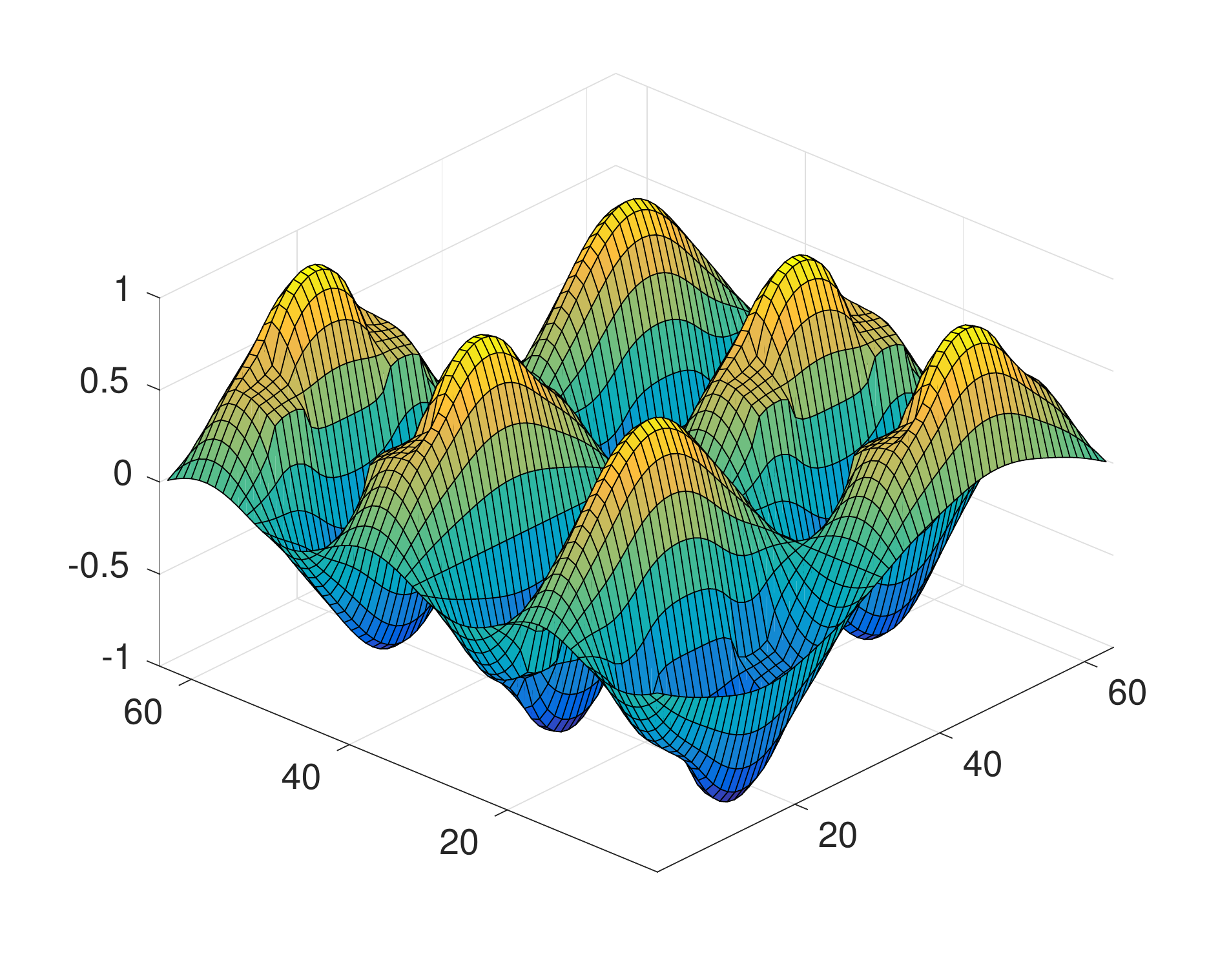} 
\caption{\small Examples of solutions $u$ for $m=255$, where $L=2$ (left) and $L=4$ (right).}
\label{fig:RHSsolut_L2}
\end{figure} 

Table \ref{Tab:conv_grids} shows the decay of the solution error in $L_2$-norm estimated 
on a sequence of dyadic refined grids, and for different values of $L = 2,4,8$. 
The solution is supposed to be represented on a sequence of grid 
in the form ${\bf u}_p= {\bf c}_0 + {\bf c}_1 h_p^\beta$ up to low order term.
We expect 
the asymptotic error behavior $O(h^{\beta})$ with $1\leq \beta \leq 2$, where in our case 
$\beta$ is close to $3/2$ that corresponds to decay factor  $2\sqrt{2} \approx 2.8 $. 
The latter can be expected in the case of reduced regularity in the solution caused by cusps 
by the multiple interior corners in the configuration of coefficient jumps. 
The  respective convergence rate in the $H^1$-norm is of the order of $O(h^{\beta-1})$.
\begin{table}[htb]
\begin{center}%
\begin{tabular}
[c]{|r|r|r|r|r|r|}%
 \hline  
 grid size   $m_p$       & $63 $ & $127 $ & $255 $ & $511 $ & $1023 $    \\
  \hline
 vector size $M_{d,p}$   & $3969$ & $16129$ & $65025$ & $261121$ & $1046529$    \\
 \hline
$L=2$   & -& $0.0051$ &$0.0015$ & $5.03e-04$ & $1.806e-04$ \\
 \hline  
 $L=4$   & -& $0.0057$ & $0.0020$ & $7.13e-04$ & $2.638e-04$ \\
 \hline  
 $L=8$    & - & -& $0.0035$ & $1.40e-03$  & $5.257e-04$ \\
 \hline 
 \end{tabular}
\caption{Differences in the relative norms of solutions 
$\|{\bf u}_p -{\bf u}_{p-1}\|_2/\|{\bf u}_{p-1}\|_2$, $p=1,\ldots,5$, on dyadic refined grids computed 
in $L_2$-norm  for $L=2,\; 4,\; 8$ and $\alpha=0.5$, $\lambda=0.1$. 
}
\label{Tab:conv_grids}
\end{center}
\end{table}
 
Figure \ref{fig:RHSsolut_L2} illustrates examples of the 
solution ${\bf u}$ discretized over $m\times m$ grid with the univariate grid size $m=255$
and for $L=2$ and $L=4$.

Figure \ref{fig:Er_discr_L4} represents differences in solutions on pair of $m \times m$ grids 
with $m=127,\; 255$ (left) and $m=511, \; 1023$ (right) for $f=sin(2\pi x) cos(6\pi y)$
and fixed $L=4$. 
One can observe the expected increase
in the approximation error towards the interior corners in the geometry specifying
the jumping coefficient function. The error decays by a factor about $10$ which
agrees with the expected decay by $2.8^2$.
For ease of comparison both solutions are interpolated onto the common grid with $m=63$.
 \begin{figure}[htbp]
\centering
\includegraphics[width=7.6cm]{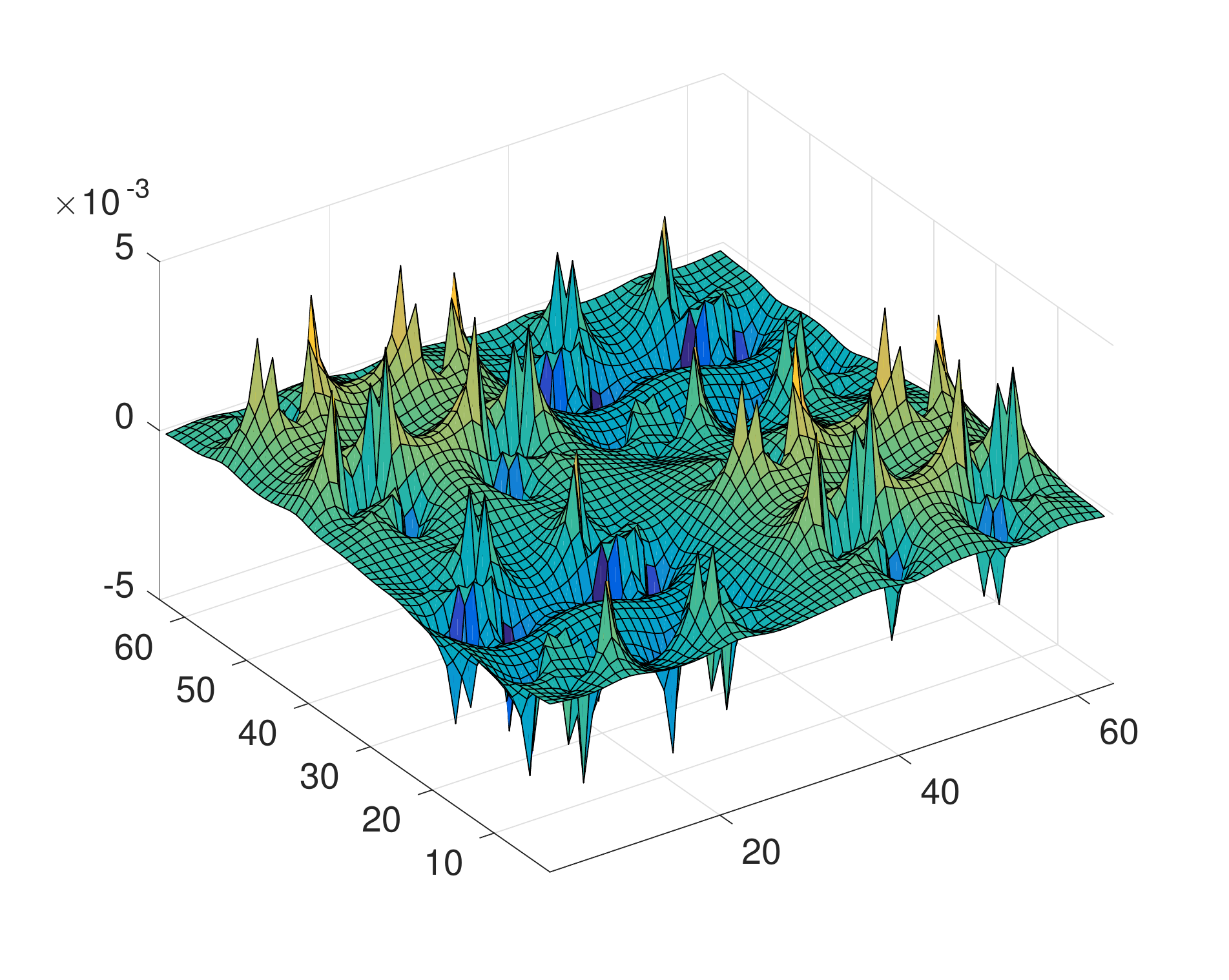}
\includegraphics[width=7.6cm]{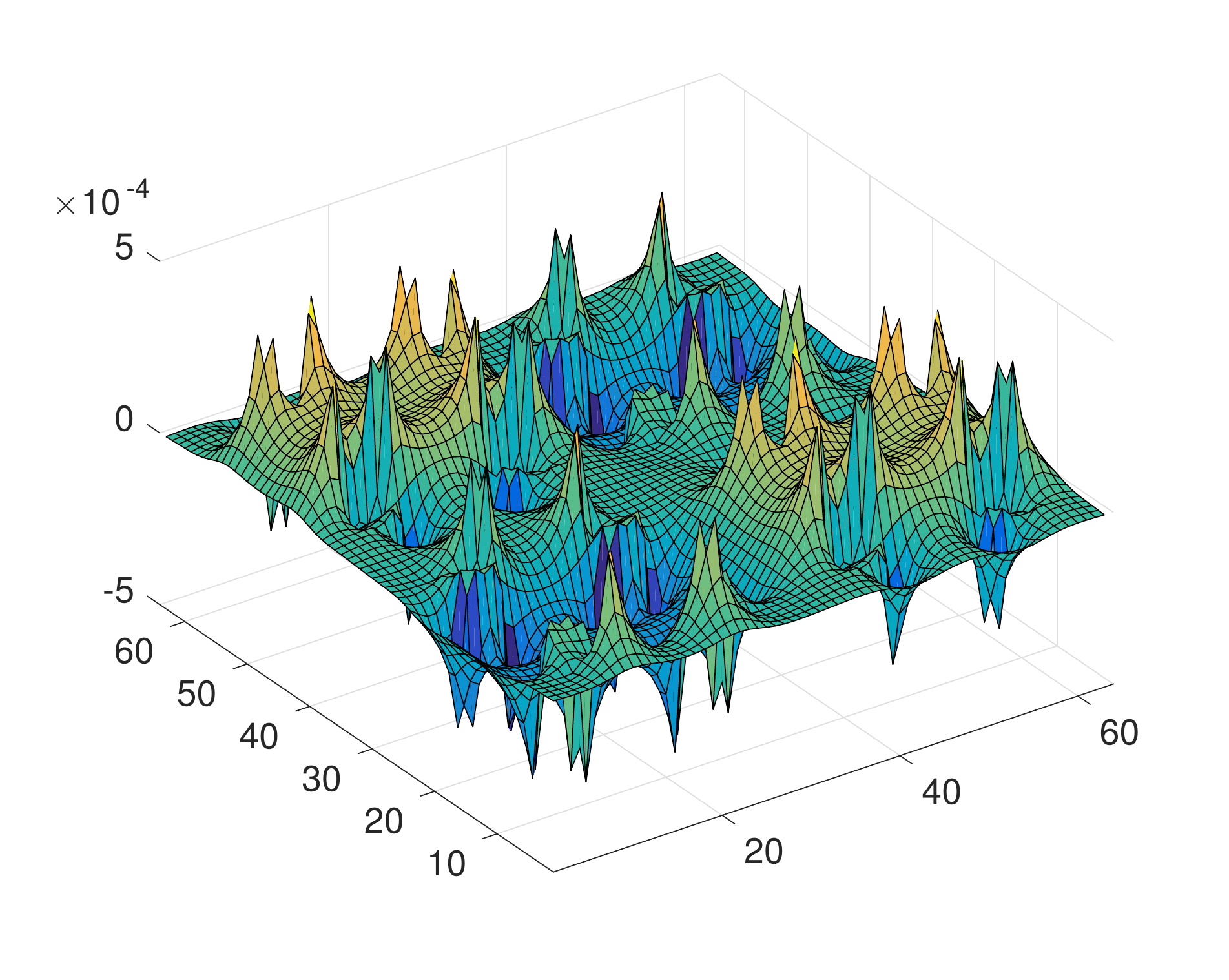} 
\caption{\small  Differences in solutions on the $m \times m$ grids with $m=127, \;255$ (left) and
$m=511, \; 1023$ (right) for  $L=4$.}
\label{fig:Er_discr_L4}
\end{figure}

\subsection{Preconditioned CG iteration}\label{ssec:PCG} 
 
 Let the right-hand side in (\ref{eqn:FEM_discr}) satisfy $\langle F, 1 \rangle=0$, 
 then for a fixed $m$, the equation 
 \begin{equation}\label{eqn:DD_Eqn}
  A^{(n)} {\bf u} =(\lambda A_\Delta + (1-\lambda) A^{(n)}_{s}){\bf u} = {\bf f}
 \end{equation}
has the unique solution. We solve this equation by the preconditioned conjugate gradient (PCG) 
iteration (routine \emph{pcg} in Matlab library) with the preconditioner
\[
 B= \frac{1+\lambda}{2} A_\Delta + \delta I= \frac{1+\lambda}{2} \Delta_h  + \delta I, 
\]
where $\delta >0$ is a small regularization parameter introduced only for  
stability reasons (can be ignored in the theory) and $I$ is the $M_d\times M_d$ identity matrix. 

It can be proven that the condition number of preconditioned matrix is uniformly 
bounded in $m_1$, $L$ and in the number of stochastic realizations $n=1,\ldots,N$.
The particular estimates on the condition number in terms of a parameter $\lambda$
can be derived by introducing the average coefficient
$$
{a}_0(x)=\frac{1}{2}(a^+(x) + a^-(x)),
$$ 
where $a^+(x)$ and $a^-(x)$ are chosen as {\it majorants and minorants} 
of $a^{(n)}(x)$ in (\ref{eqn:Am_coef}), respectively. The following simple result holds.
\begin{lemma}\label{lem:Cond_Number}
 Given the preconditioner $B$ with $\delta=0$, then
 the condition number of the preconditioned matrix
$B^{-1} {A}^{(n)}$ is bounded by
\[
 cond \{ B^{-1}{A}^{(n)}  \} \leq  C{\lambda}^{-1}.
\]
\end{lemma}
\begin{proof}
Lemma 4.1 in \cite{BokhSRep:15} shows that the preconditioner $A_0$ generated by the coefficient
${a}_0(x)=\frac{1}{2}(a^+(x) + a^-(x))$ allows the condition number estimate
\[
 cond \{{A}_0^{-1} {A}^{(n)}  \} \leq C \max\frac{1+q}{1-q}, \quad\mbox{with}\quad
 q:=\max(a^+(x) - {a}_0(x))/{a}_0(x)<1.
\]
 The preconditioner $B$ corresponds to the choice $a^+(x)=1$ and $a^-(x)=\lambda$,
 hence, we obtain ${a}_0(x)=\frac{1+\lambda}{2}$ and the result follows.
\end{proof}
 
The PCG solver for the system of equations (\ref{eqn:FEM_syst}) with the shifted discrete Laplacian 
as the preconditioner demonstrates robust convergence with the rate $q\ll 1$.
In the practically interesting case $\alpha \approx 0.5$ we found that 
$q$ does not depend on $\lambda$. This can be explained by the fact that in this case 
the total overlap in all subdomains covers the large portion of the computational box $\Omega$.
In all numerical examples considered so far the number of PCG iterations was smaller than $10$
for the residual stopping criteria $\delta=10^{-8}$. 
We use the univariate grid size $m_1=m_s$, corresponding to the 
choice $p=0$ in (\ref{eqn:Grid_relation}) which is fine enough to resolve geometry for larger $L$.

 \section{Asymptotic convergence to the stochastic average}
 \label{sec:Comp_Ahom}
 
 In this section, we describe the computational scheme for calculation of 
 the homogenized coefficient matrix for each stochastic realization.
 
 \subsection{Computational scheme for the stochastic average}\label{ssec:SchemeAverage}
 
  For fixed stochastic realizations specifying the variable part in the $2\times 2$ coefficient matrix
 $\widehat{\mathbb{A}}^{(n)}(x)$, $n=1,\ldots,N$, we consider the problems 
 \begin{equation}\label{eqn:RHS_grad}
 -\lambda \Delta\phi_i -(1-\lambda)\nabla \cdot \widehat{\mathbb{A}}^{(n)}(\cdot)
 ({\bf e}_i +\nabla \phi_i )=0,
 \end{equation}
 for $i=1,2$. 
 The right-hand side in equation (\ref{eqn:RHS_grad}), 
 rewritten in the canonical form (\ref{eqn:scaled_setting}), reads as 
 \[
  f_i(x)= (1-\lambda) \nabla \cdot \widehat{\mathbb{A}}^{(n)}(x){\bf e}_i.
 \]  
Taking into account (\ref{eqn:Am_coef}), where the diagonal of $\widehat{\mathbb{A}}^{(n)}(x)$ 
is defined in terms of the scalar function $\widehat{a}^{(n)}(x)$, we arrive at
\[
 f_1(x)= (1-\lambda)\dfrac{\partial \widehat{a}^{(n)}(x)}{\partial x_1}, \quad
 f_2(x)=(1-\lambda)\dfrac{\partial \widehat{a}^{(n)}(x)}{\partial x_2}.
\]
 Figure \ref{fig:RHSsolutA} illustrates
an example of the calculated (reshaped) right-hand side vector in $\mathbb{R}^{m_1\times m_1}$ 
and the respective solution
$\phi_1$ for $L=12$, $m_1=193$, and $m_0=16$. 
\begin{figure}[htbp]
\centering
\includegraphics[width=7.6cm]{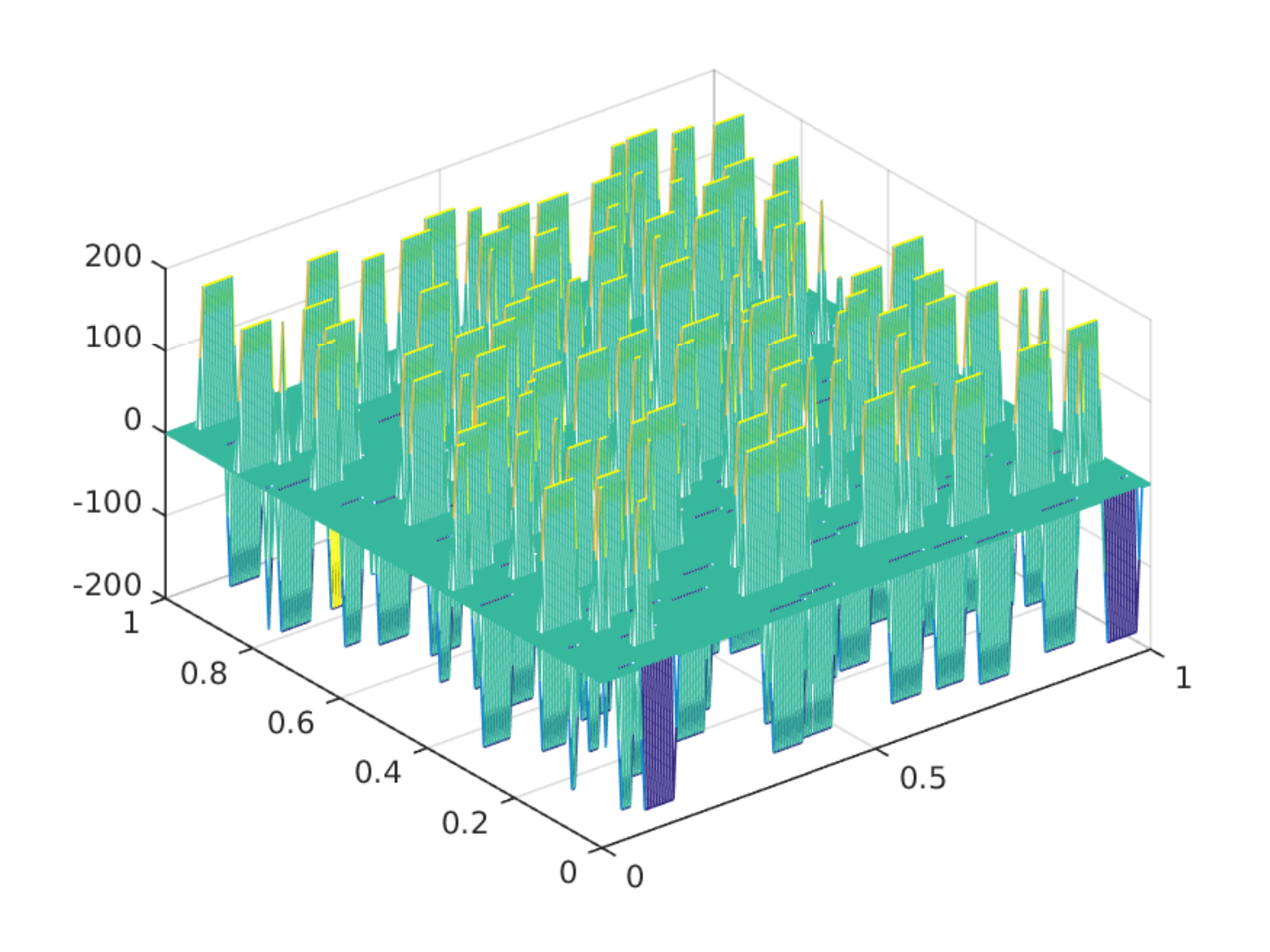}
\includegraphics[width=7.6cm]{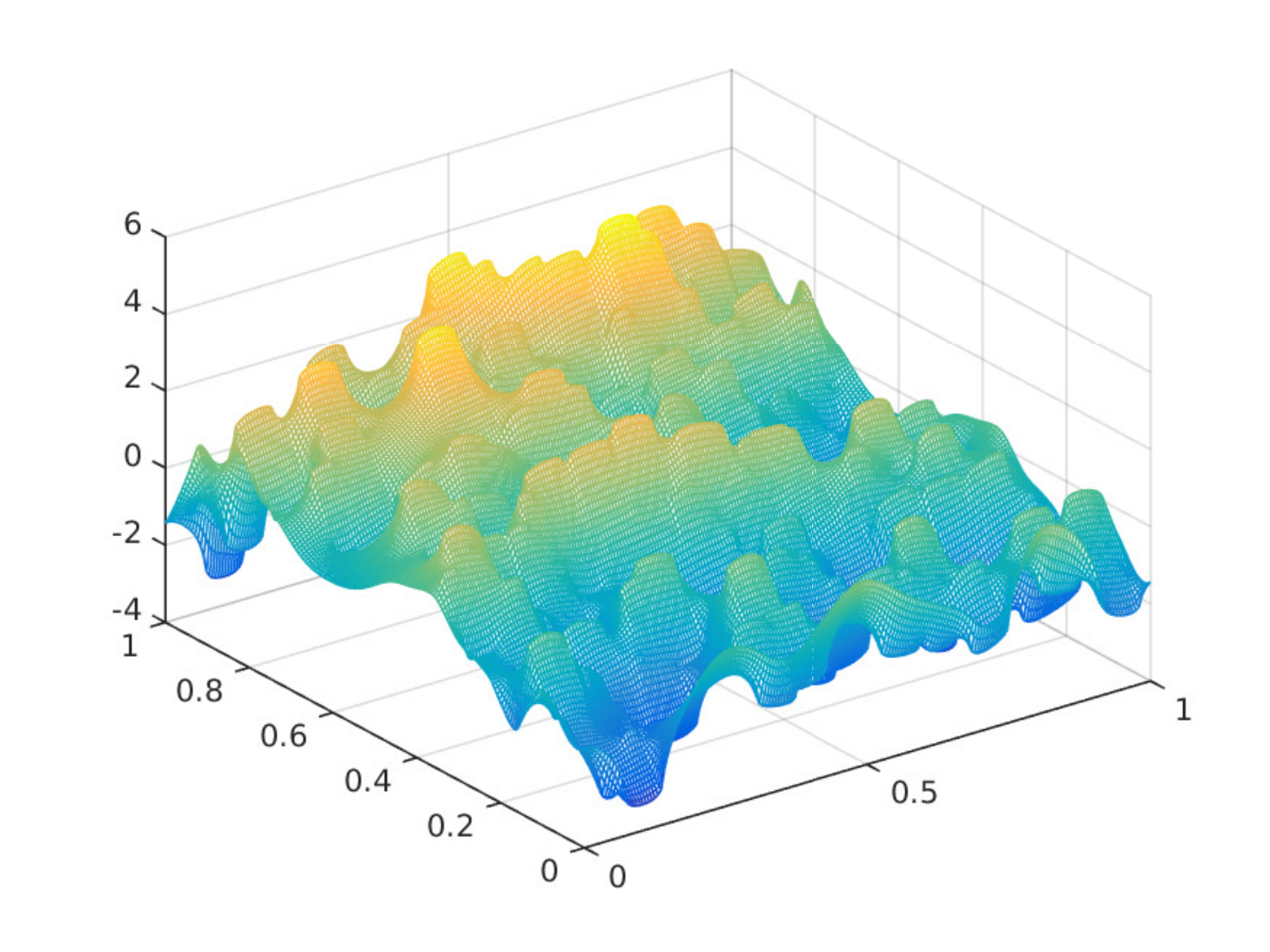} 
\caption{\small Right-hand side (left) and the solution  $\phi_1$ (right) for 
$L=12$, $m_1=193$, $m_0=16$.}
\label{fig:RHSsolutA}
\end{figure}

Fixed $L$, for the particular realization $\mathbb{A}^{(n)}$, by definition, the averaged coefficient matrix 
$\bar{\mathbb{A}}^{(n)}_L = \bar{\mathbb{A}}^{(n)} =[\bar{a}^{(n)}_{ij}]\in \mathbb{R}^{2\times 2}$, 
$i,j=1,2$, with the constant entries is given by
\begin{equation}\label{eqn:A_hom_m}
 \bar{\mathbb{A}}^{(n)}  {\bf e}_i =\int_\Omega \mathbb{A}^{(n)} (x)({\bf e}_i + 
 \nabla \phi_i ) dx,
\end{equation}
which implies the representation for matrix elements
\[
 \bar{a}^{(n)}_{L,ij} \equiv \bar{a}^{(n)}_{ij} = 
 \int_\Omega [(\lambda I_{2\times 2} +   (1-\lambda)\widehat{\mathbb{A}}^{(n)} (x))
 ({\bf e}_i  + \nabla \phi_i )]_j) dx,  \quad i,j=1,2.
 \]
The latter leads to the entry-wise representation of the matrix 
$\bar{\mathbb{A}}^{(n)} =[\bar{a}^{(n)}_{i j}]$, $i, j=1,2$,
\begin{align} \label{eqn:Ahom_m_entry}     
    &   \bar{a}^{(n)}_{11}= \int_\Omega a^{(n)}(x) \Bigl(\dfrac{\partial \phi_1}{\partial x_1} +1\Bigr) dx,
    \nonumber \\
    & \bar{a}^{(n)}_{12}= \int_\Omega a^{(n)}(x) \dfrac{\partial \phi_1}{\partial x_2}  dx,
 \nonumber \\
    & \bar{a}^{(n)}_{21}= \int_\Omega a^{(n)}(x) \dfrac{\partial \phi_2}{\partial x_1}  dx,
  \nonumber \\
 & \bar{a}^{(n)}_{22}= 
 \int_\Omega a^{(n)}(x) \Bigl(\dfrac{\partial \phi_2}{\partial x_2} +1 \Bigr) dx.
\end{align}
The representation (\ref{eqn:Ahom_m_entry}) ensures the symmetry of the homogenized matrix 
$\bar{\mathbb{A}}^{(n)}$, i.e. $\bar{a}^{(n)}_{ij}=\bar{a}^{(n)}_{ji}$. 
Indeed, we calculate the difference between
the scalar product of the first equation in 
(\ref{eqn:RHS_grad}) with $\phi_2$,
\begin{equation*}\label{eqn:RHS_gradP2}
 \langle \lambda \nabla \phi_1 + (1-\lambda) \widehat{\mathbb{A}}^{(n)} \nabla \phi_1,\nabla \phi_2 \rangle
 -(1-\lambda) \langle\nabla \cdot \widehat{\mathbb{A}}^{(n)}(\cdot){\bf e}_1,\phi_2 \rangle=0,
 \end{equation*}
and the second equation in (\ref{eqn:RHS_grad}) with $\phi_1$, 
\begin{equation*}\label{eqn:RHS_gradP1}
 \langle \lambda \nabla \phi_2 + (1-\lambda) \widehat{\mathbb{A}}^{(n)} \nabla \phi_2,\nabla \phi_1 \rangle
 -(1-\lambda) \langle\nabla \cdot \widehat{\mathbb{A}}^{(n)}(\cdot){\bf e}_2,\phi_1 \rangle=0,
 \end{equation*}
and get the relation
\[
  \langle \dfrac{\partial \widehat{a}^{(n)}}{\partial x_1}, \phi_2 \rangle
  -\langle  \dfrac{\partial \widehat{a}^{(n)}}{\partial x_2}, \phi_1  \rangle =0,
\]
which then implies the desired property via integration by parts, and 
taking into account the relation (\ref{eqn:coef_diag}),
\[
 \langle  a^{(n)}, \dfrac{\partial \phi_2}{\partial x_1}  \rangle
 = \langle a^{(n)}, \dfrac{\partial \phi_1}{\partial x_2}  \rangle.
\]

In numerical implementation, we apply the Galerkin  scheme for FEM discretization 
of equation (\ref{eqn:RHS_grad}) its right-hand side. 
We use the same quadrature rule for computation of integrals in (\ref{eqn:Ahom_m_entry})
thus preserving the symmetry in the matrix 
$\mathbb{A}^{(n)}$ inherited from the exact variational formulation 
(see argument above and Section \ref{ssec:Symmetr_Quartic_tensQ} 
for the more detailed discussion).
 
Integrals over $\Omega$ in (\ref{eqn:A_hom_m}), (\ref{eqn:Ahom_m_entry})
for the  matrix entries $(\bar{\mathbb A}^{(n)})_{i,j}$,  $i,j=1,2$,
are calculated (approximately) by the scalar 
product of the $N$-vector of all-ones with the discrete representation of integrand
on the grid $\Omega_h$, see Figure \ref{fig:CovDomain}.

\begin{table}[htb]
\begin{center}%
{\footnotesize
\begin{tabular}
[c]{|r|r|r|r|r|r|r|r|r|r|}%
 \hline
 Tol. $\delta$   & $10^{-3}$ & $10^{-4}$ & $10^{-5}$ & $10^{-6}$ & $10^{-7}$ &$10^{-8}$ &
 $10^{-9}$ & $10^{-10}$ &$10^{-11}$   \\
  \hline
 $\| \mathbb{A}- \mathbb{A}^T  \|$  &$10^{-6}$ & $3 \;10^{-7}$ & $ 10^{-7}$ & 
 $3\;10^{-9}$&$10^{-10}$&$3.6\;10^{-11}$&$10^{-11}$&
 $10^{-12}$ & $5.7\;10^{-15}$ \\
 \hline
 \end{tabular}
\caption{\small Symmetry in the matrix $\bar{\mathbb{A}}^{(n)}_L$, with fixed $n$, vs. residual stopping 
criteria $\delta$.}
\label{Tab:conv_Symmet_thresh}
}
\end{center}
\end{table}
To complete this section, we check numerically that the FEM discretization scheme preserves the symmetry in 
the matrix $\bar{\mathbb{A}}^{(n)}_L$ for fixed $L$ 
if the discrete system of equations (\ref{eqn:FEM_Galerk}) is solved accurately enough.
Table  \ref{Tab:conv_Symmet_thresh} demonstrates that the symmetry in the matrix 
$\bar{\mathbb{A}}^{(n)}_L$ with fixed $n$ is recovered on the level of residual stopping criteria $\delta>0$ in 
the preconditioned iteration for solving the discrete system of equations. 
For this calculation we set $L=4$, $m_0=8$, $\alpha=0.5$ and $\lambda=0.2$.

\subsection{Asymptotic of systematic error and standard deviation}
\label{ssec:SystErStandDev}

The set of numerical approximations $\{\bar{\mathbb{A}}^{(n)}_L\}$
to the homogenized matrix $\mathbb{A}_{\mbox{\footnotesize hom}}$ is calculated by 
(\ref{eqn:A_hom_m}) for the sequence $\{\mathbb{A}^{(n)}_L(x)\}$ of $n=1,...,N$ realizations, 
where $N$ is large enough,
and the artificial period $L$ defines the size of Representative Volume Elements (RVE). 
For a fixed $L$, the approximation  $\bar{\mathbb{A}}^{N}_L$ is computed 
as the \emph{empirical average}  of the sequence $\{\bar{\mathbb{A}}^{(n)}_L\}_{n=1}^N$,
\begin{equation} \label{eqn:A_hom_Empir}
 \bar{\mathbb{A}}^{N}_L= \frac{1}{N}\sum^{N}_{n=1} \bar{\mathbb{A}}^{(n)}_L.
\end{equation}

By the law of large numbers we have that the empirical average converges 
almost surely to the \emph{ensemble average} (expectation)
\begin{equation} \label{eqn:A_hom_Ensemble}
 \langle  \bar{\mathbb{A}}_L\rangle_L =  
\lim\limits_{N \to \infty} \bar{\mathbb{A}}^{N}_L. 
\end{equation}
Furthermore, by qualitative homogenization theory, as the artificial period $L\to \infty$,
this converges to the homogenized matrix
\begin{equation} \label{eqn:A_hom}
 \mathbb{A}_{\mbox{\footnotesize hom}}: = \lim\limits_{L\to \infty} \langle  \bar{\mathbb{A}}_L\rangle_L .
\end{equation}
In what follows, we use the entry-wise notation for $d\times d$ 
matrices ${\mathbb{A}}=[a_{ij}]$, $i,j=1,\ldots,d$,
for example, $\langle \bar{\mathbb{A}}_L \rangle =[\bar{a}_{L,ij}]$ and 
$\bar{\mathbb{A}}_L^{(n)}=[\bar{a}_{L,ij}^{(n)}]$, etc.

In terms of square expectations, the convergence rate for the computable quantities 
can be estimated by, see \cite{GlNeuOtto:13},
\begin{equation}
\label{eqn:CLT_RVE}
 \left\langle |\bar{\mathbb{A}}^{N}_L - \mathbb{A}_{\mbox{\footnotesize hom}}|^2 \right\rangle_L^{1/2} \leq  
  \frac{C_1}{\sqrt{N}}L^{-d/2}  + C_2 L^{-d}  \log^d L.
\end{equation}
We  numerically study  the asymptotic of both terms on the right-hand side 
 of (\ref{eqn:CLT_RVE}) separately by considering the \emph{random part} of the error,
 \begin{equation}
\label{eqn:var}
  var^{1/2}_L (\bar{\mathbb{A}}_L ) = 
  \langle \left|\mathbb{A}_{\mbox{\footnotesize hom}} -\langle \bar{\mathbb{A}}_L\rangle_L \right|^2 
  \rangle_L^{1/2} \leq C_1L^{-d/2},
 \end{equation}
and the \emph{systematic error}
 \begin{equation}
\label{eqn:sysErr}
\left|\mathbb{A}_{\mbox{\footnotesize hom}}- \langle \bar{\mathbb{A}}_L\rangle_L \right| \leq C_2 L^{-d} \log^d L,
\end{equation}
where $\langle \bar{\mathbb{A}}_L\rangle_L$ is calculated for large enough $N$.

\subsection{Covariances of the homogenized matrix in the form of quartic tensor }
\label{ssec:Symmetr_Quartic_tensQ}

Let $\langle\cdot\rangle_L$ be an ensemble of uniformly elliptic symmetric coefficient fields
on the $d$-dimensional torus $[0,L)^d$. Assume that it is invariant under translation (stationary)
and under the group ${\mathcal G}$ of all orthogonal transformations $R$ of $\mathbb{R}^d$ that leave 
the (hyper-)cube $[0,L)^d$ invariant 
(this is generated by rotations in one of the Cartesian two-dimensional planes 
and reflections along any Cartesian hyper-plane) in the sense of (\ref{wg05}) below.
In case of isotropic (i.e. scalar) coefficient fields $\mathbb{A}(x)$, (\ref{wg05}) turns into
\begin{align*}
\mathbb{A}(R\cdot)\;\mbox{and}\;\mathbb{A}\;\mbox{have the same distribution under}\;\langle\cdot\rangle_L,
\end{align*}
which is certainly the case for the ensembles we 
consider numerically.


Let $X$ be a finite-dimensional space of functions on the torus $[0,L)^d$
of side-length $L$ with square-integrable gradients, e.g. coming from continuous, 
piecewise affine Finite Elements. 
For a given realization $\mathbb{A}(x) =\mathbb{A}^{(n)}(x)$ of the coefficient field and any
direction $i=1,\cdots,d$, we consider $\phi_i\in X$ defined through 
\begin{align}\label{wg03} 
\forall\;\psi\in X\quad\int_{[0,L)^d}\nabla\psi\cdot \mathbb{A}({\bf e}_i+\nabla\phi_i)=0,
\end{align}
where ${\bf e}_i$ denotes the unit vector in direction $i$. If $X$ contains the constant functions
(as would be the case for the Finite Element space), $\phi_i$ has to be normalized to be unique, e.g.
by imposing $\int_{[0,L)^d}\phi_i=0$, but this should be irrelevant since we are only interested in
$\nabla\phi_i$. If $X$ is indeed a Finite Element space, and if 
$\{\psi_\alpha\}_{\mbox{nodes}\;\alpha}$ denotes 
the standard ``H\"utchen'' basis then the (stiffness) matrix $A=[a_{\alpha\beta}]$ and the 
right hand side ${\bf f}=[f_\alpha]$ are given by
\begin{align}
a_{\alpha\beta}=\int_{[0,L)^d}\nabla\psi_\alpha\cdot \mathbb{A}\nabla\psi_\beta
\quad\mbox{and}\quad f_\alpha=-\int_{[0,L)^d}\nabla\psi_\alpha\cdot \mathbb{A} {\bf e}_i.
\end{align}
Here it is important to treat periodicity correctly: In practice, one identifies functions on $[0,L)^d$
with functions on $\mathbb{R}^d$ that are periodic in each (Cartesian) argument of period $L$,
hence if the node $\alpha$ is such that one of the adjacent triangles crosses the boundary of 
the periodic cell $[0,L)^d\subset\mathbb{R}^d$, then there is a piece of $\phi_\alpha$ 
that appears on the other side.
If a quadrature rule is used for computing the stiffness matrix, it is important 
that the same one is used for approximation of the right-hand side.


Let us consider the $d\times d$ matrix $\bar{\mathbb{A}}_L=[\bar a_{L,ij}]=\bar{\mathbb{A}}_L(\mathbb{A})$ 
defined through (see also (\ref{eqn:A_hom_m}))
\begin{align}\label{wg01}
\bar a_{L,ij}:= {\bf e}_j\cdot\int_{[0,L)^d}\mathbb{A}({\bf e}_i+\nabla\phi_i),
\end{align}
(where again, the same quadrature rule should be used). Then we have for every realization
\begin{align}\label{wg02}
 \bar{\mathbb{A}}_L\;\;\mbox{is symmetric, i.e.}\;\; \bar a_{L,ij}=\bar a_{L,ji}.
\end{align}


Let us consider the ensemble average $\langle \bar{\mathbb{A}}_L \rangle_L$, 
which by the law of large numbers is given by (see also (\ref{eqn:A_hom_Ensemble}))
\begin{align}
\langle \bar{\mathbb{A}}_L\rangle_L=\lim_{N\uparrow\infty}\frac{1}{N}\sum_{n=1}^N   \bar{\mathbb{A}}_L^{(n)},
\end{align}
almost surely,
where $\bar{\mathbb{A}}_L^{(n)}$ come via (\ref{wg01}) from independent realizations 
$\mathbb{A}=\mathbb{A}^{(n)}$ according to the distribution $\langle\cdot\rangle_L$.
Suppose that the finite-dimensional space $X$ is invariant under reflections in the coordinate directions
in the sense of (\ref{wg07}) below. This imposes a more serious restriction on the Finite Element space,
namely that it is based on a subdivision of the torus $[0,L)^d$ into axi-parallel 
cubes (instead of triangles) and that 
the function space on each cube is spanned by functions that are multi-linear in the 
Cartesian coordinates (as opposed
to affine). If this condition is satisfied, then we have
\begin{align}\label{wg04}
\langle \bar{\mathbb{A}}_L\rangle_L \;\mbox{is isotropic, i.e.}\;
\langle \bar a_{L,ij}\rangle_L=\lambda_L\delta_{ij}
\end{align}
for some $\lambda_L\in(0,L)$.


We are interested in the covariances of the entries of $\bar{\mathbb{A}}_L$, and note that
by the law of large numbers
\begin{align*}
\lefteqn{{\rm cov}_{\langle\cdot\rangle_L}[\bar a_{L,ij},\bar a_{L,i'j'}]
:=\big\langle(\bar a_{L,ij}-\langle \bar a_{L,ij}\rangle_L)(\bar a_{L,i'j'}-\langle 
\bar a_{L,i'j'}\rangle_L)\big\rangle_L}
\nonumber\\
&=\lim_{N\uparrow\infty}\frac{1}{N-1}\sum_{n=1}^N\big(\bar a_{L,ij}^{(n)}-\frac{1}{N}\sum_{m=1}^N\bar a_{L,ij}^{(m)}\big)
\big(\bar a_{L,i'j'}^{(n)}-\frac{1}{N}\sum_{m'=1}^N \bar a_{L,i'j'}^{(m')}\big).
\end{align*}
More precisely, we are interested in its rescaled version
\begin{align*}
\bar Q_{L,iji'j'}:=L^d{\rm cov}_{\langle\cdot\rangle_L}[\bar a_{L,ij},\bar a_{L,i'j'}]
\end{align*}
which is easier to understand as the four-linear form
\begin{align*}
\bar Q_{L}(\eta,\xi,\eta',\xi'):=
L^d{\rm cov}_{\langle\cdot\rangle_L}[\eta\cdot \bar{\mathbb{A}}_L \xi,\eta'\cdot \bar{\mathbb{A}}_L\xi'].
\end{align*}
We claim that it has the invariance property
\begin{align}\label{wg13}
\bar Q_{L}(R\eta,R\xi,R\eta',R\xi')=\bar Q_{L}(\eta,\xi,\eta',\xi').
\end{align}
In the case of $d=2$, this implies that $\bar Q_L$ is just characterized by three different numbers:
\begin{align}
\bar Q_L(e_1,e_1,e_1,e_2)&=\bar Q_L(e_1,e_1,e_2,e_1)=\nonumber\\
\bar Q_L(e_1,e_2,e_1,e_1)&=\bar Q_L(e_2,e_1,e_1,e_1)=0,\label{wg16}\\
\bar Q_L(e_1,e_2,e_2,e_2)&=\bar Q_L(e_2,e_1,e_2,e_2)=\nonumber\\
\bar Q_L(e_2,e_2,e_1,e_2)&=\bar Q_L(e_2,e_2,e_2,e_1)=0,\label{wg17}\\
\bar Q_L(e_1,e_2,e_1,e_2)&=\bar Q_L(e_1,e_2,e_2,e_1)=\nonumber\\
\bar Q_L(e_2,e_1,e_1,e_2)&=\bar Q_L(e_2,e_1,e_2,e_1),\label{wg14}\\
\bar Q_L(e_1,e_1,e_2,e_2)&=\bar Q_L(e_2,e_2,e_1,e_1),\label{wg15}\\
\bar Q_L(e_1,e_1,e_1,e_1)&=\bar Q_L(e_2,e_2,e_2,e_2).\label{wg18}
\end{align}

\medskip

{\sc Argument for (\ref{wg02})}. According to (\ref{wg03}), definition (\ref{wg01}) may be reformulated as
\begin{align*}
\bar a_{L,ij}=\int_{[0,L)^d}({\bf e}_j+\nabla\phi_j)\cdot \mathbb{A} ({\bf e}_i+\nabla\phi_i),
\end{align*}
so that the symmetry of $\mathbb{A}$ yields the symmetry of $\bar{\mathbb{A}}_L$.

\medskip

{\sc Argument for (\ref{wg04})}. Identifying the points on the torus with $[0,L)^d\subset\mathbb{R}^d$,
let ${\mathcal G}$ denote the subgroup of the orthogonal group that leaves $[0,L)^d$ invariant.
According to our assumption, for any $R\in{\mathcal G}$,
\begin{align}\label{wg05}
R^t \mathbb{A}(R\cdot)R\;\mbox{and}\;\mathbb{A}\;\mbox{have the same distribution under}\;\langle\cdot\rangle_L,
\end{align}
where $R^t \mathbb{A}(R\cdot)R$ denotes the matrix field $[0,L)^d\ni x\mapsto R^t \mathbb{A}(R x)R$. 
According to our assumption on $X$ we have
\begin{align}\label{wg07}
\psi \in X\;\Longrightarrow\;\psi (R\cdot)\in X,
\end{align}
where $\psi(R\cdot)$ denotes the function $[0,L)^d\ni x\mapsto \psi(Rx)$. 

\smallskip

For a fixed vector $\xi\in\mathbb{R}^d$, we consider $\phi_\xi:=\xi_i\phi_i$ (Einstein's summation convention)
and note that in view of (\ref{wg03}), for given realization $\mathbb{A}=\mathbb{A}^{(n)}$,
the function $\phi_\xi=\phi_\xi(\mathbb{A})$ (at least up to additive constants) is characterized by 
\begin{align}\label{wg08}
\forall\;\psi \in X\quad\int_{[0,L)^d}\nabla\psi\cdot \mathbb{A}(\xi+\nabla\phi_\xi)=0,
\end{align}
We now argue that $\phi$ transforms under $R\in{\mathcal G}$ as follows
\begin{align}\label{wg10}
\phi_{R\xi}(\mathbb{A};Rx)=\phi_\xi(R^t \mathbb{A} (R\cdot)R;x).
\end{align}
Indeed, this relies on the straightforward orthogonal transformation rule
\begin{align*}
\lefteqn{\int_{[0,L)^d}\nabla_y[\psi (Ry)]\cdot \mathbb{A}(y)(R\xi+\nabla\phi_{R\xi}(y))dy}\nonumber\\
&\stackrel{y=Rx}{=}\int_{[0,L)^d}\nabla\psi(x)\cdot R^t \mathbb{A}(Rx)R(\xi+\nabla_x[\phi_{R\xi}(Rx)])dx.
\end{align*}
According to (\ref{wg07}) and (\ref{wg08}) (with $\xi$ replaced by $R\xi$) the left-hand side
vanishes for all $\psi \in X$;
hence by the characterization (\ref{wg08}) applied to the right-hand side, we obtain (\ref{wg10}).

\smallskip

We now argue note that from (\ref{wg10}) we obtain for the gradient
$\nabla\phi_{R\xi}(\mathbb{A};Rx)=\nabla\phi_\xi(R^t \mathbb{A}(R\cdot)R;x)$ and thus for the 
flux $q_\xi(\mathbb{A};x):= \mathbb{A}(\xi+\nabla\phi_\xi(\mathbb{A};x))$ the transformation rule
\begin{align*}
q_{R\xi}(\mathbb{A};Rx)=Rq_\xi(R^t \mathbb{A}(R\cdot)R;x),
\end{align*}
from which we obtain by definition (\ref{wg01}) that
\begin{align}\label{wg19}
\bar{\mathbb{A}}_L(\mathbb{A})R\xi=R \bar{\mathbb{A}}_L(R^t \mathbb{A}(R\cdot)R)\xi.
\end{align}
According to (\ref{wg05}) this yields the following invariance property for the symmetric
matrix $\langle \bar{\mathbb{A}}_L\rangle_L$
\begin{align*}
\langle \bar{\mathbb{A}}_L\rangle_L R\xi=R\langle  \bar{\mathbb{A}}_L \rangle_L\xi.
\end{align*}
Since this holds for all $\xi\in\mathbb{R}^d$ and all $R\in{\mathcal G}$, by an argument of elementary algebra,
we obtain the isotropy of $\langle \bar{\mathbb{A}}_L \rangle_L$, cf (\ref{wg04}).

\medskip

{\sc Argument for (\ref{wg13})}. This follows from (\ref{wg19}) in form of 
$$
(R\eta)\cdot \bar{\mathbb{A}}_L (\mathbb{A})R\xi=\eta\cdot \bar{\mathbb{A}}_L(R^t \mathbb{A}(R\cdot)R)\xi
$$
and from (\ref{wg05}).

\medskip

{\sc Argument for (\ref{wg18})-(\ref{wg16})}. The four identities in (\ref{wg14}) on 
the variances just follow from 
the symmetry of the underlying random variable $\bar{\mathbb{A}}_L$, c.f. ~(\ref{wg02}), 
in form of 
$$
{\bf e}_1\cdot\bar{\mathbb{A}}_L {\bf e}_2={\bf e}_2\cdot\bar{\mathbb{A}}_L {\bf e}_1.
$$ 
The identity (\ref{wg15}) follows from the symmetry of the covariance in its two arguments.
The vanishing of the eight entries stated in (\ref{wg16}) and (\ref{wg17}) follows from (\ref{wg13}) 
applied to the reflection $R\in{\mathcal G}$ given by $R {\bf e}_1=-{\bf e}_1$ and 
$R {\bf e}_2={\bf e}_2$. The identity (\ref{wg18})
follows from (\ref{wg13}) applied to the reflection $R\in{\mathcal G}$ given by 
$R {\bf e}_1={\bf e}_2$ and $R{\bf e}_2={\bf e}_1$.

\section{Numerical study of stochastic homogenization }
\label{sec:Homo_Numer}

In this section, we estimate numerically the mean constant coefficient in 
 the system (\ref{eqn:scaled_setting})
 depending on $L$ and other model parameters at the limit of $N \to \infty$, 
 see \cite{GlOtto:12,GlOtto:16} for the respective problem setting.

 \begin{figure}[htb]
\centering
\includegraphics[width=5.2cm]{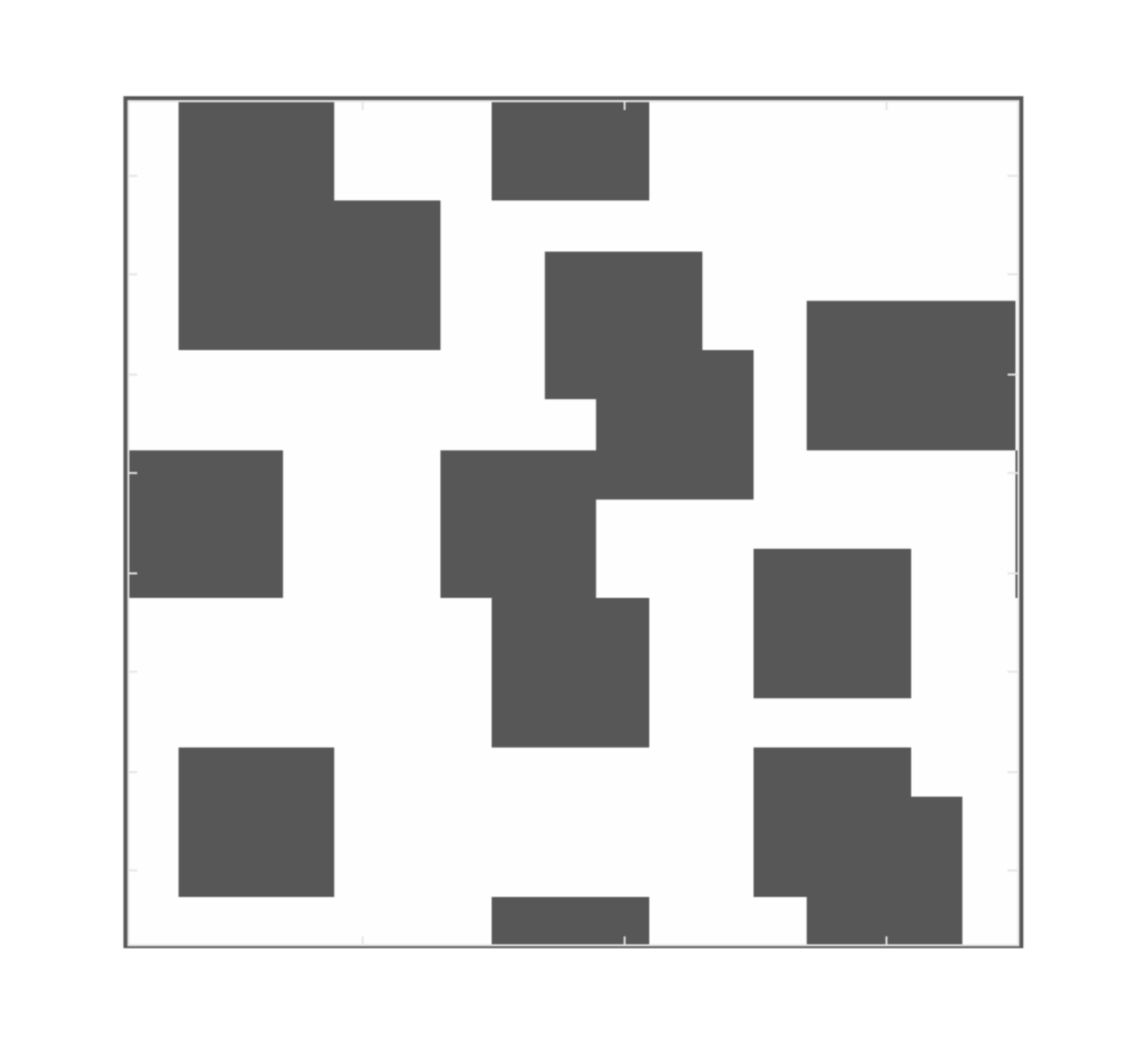}
\includegraphics[width=5.2cm]{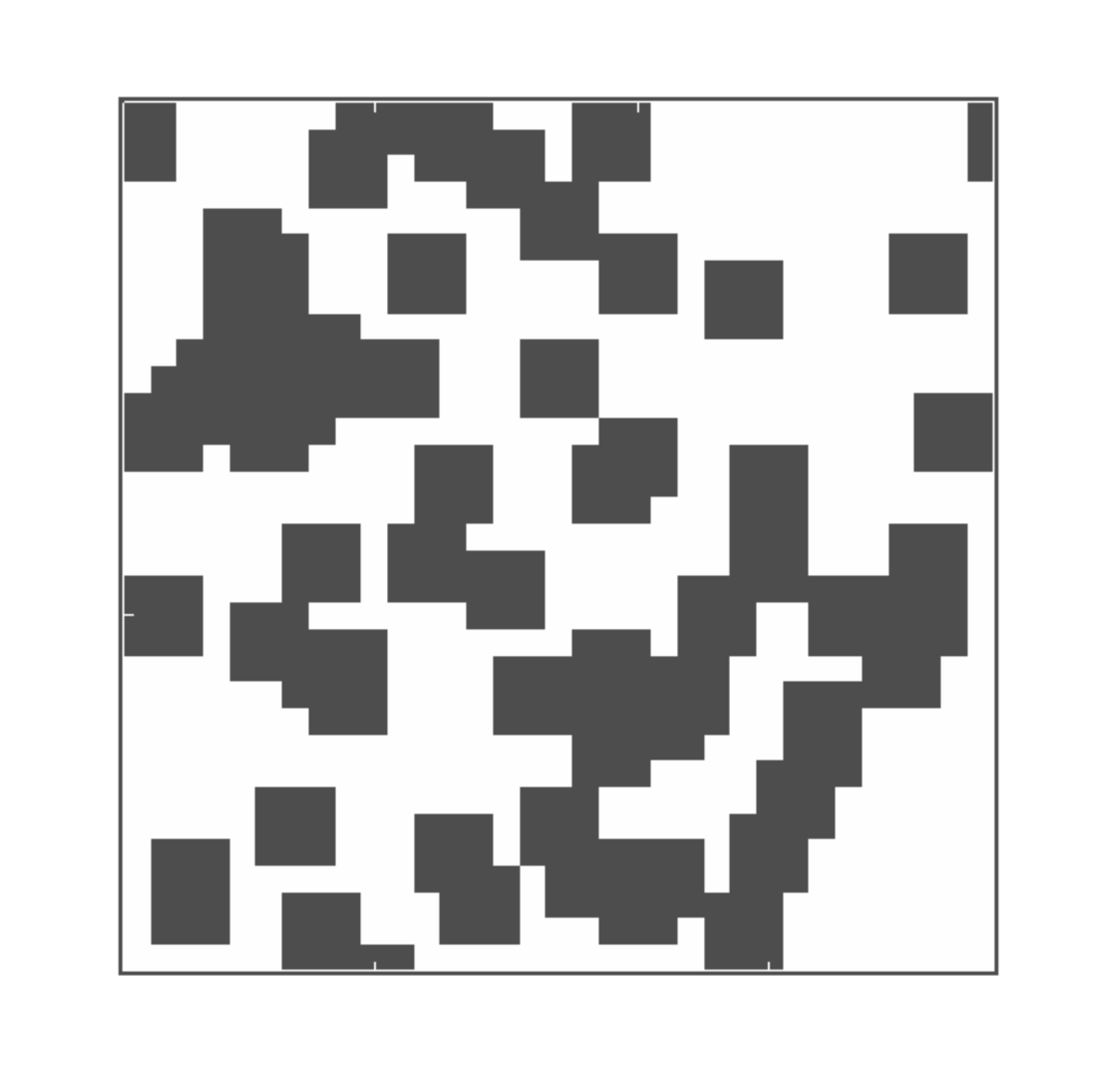}
\includegraphics[width=5.2cm]{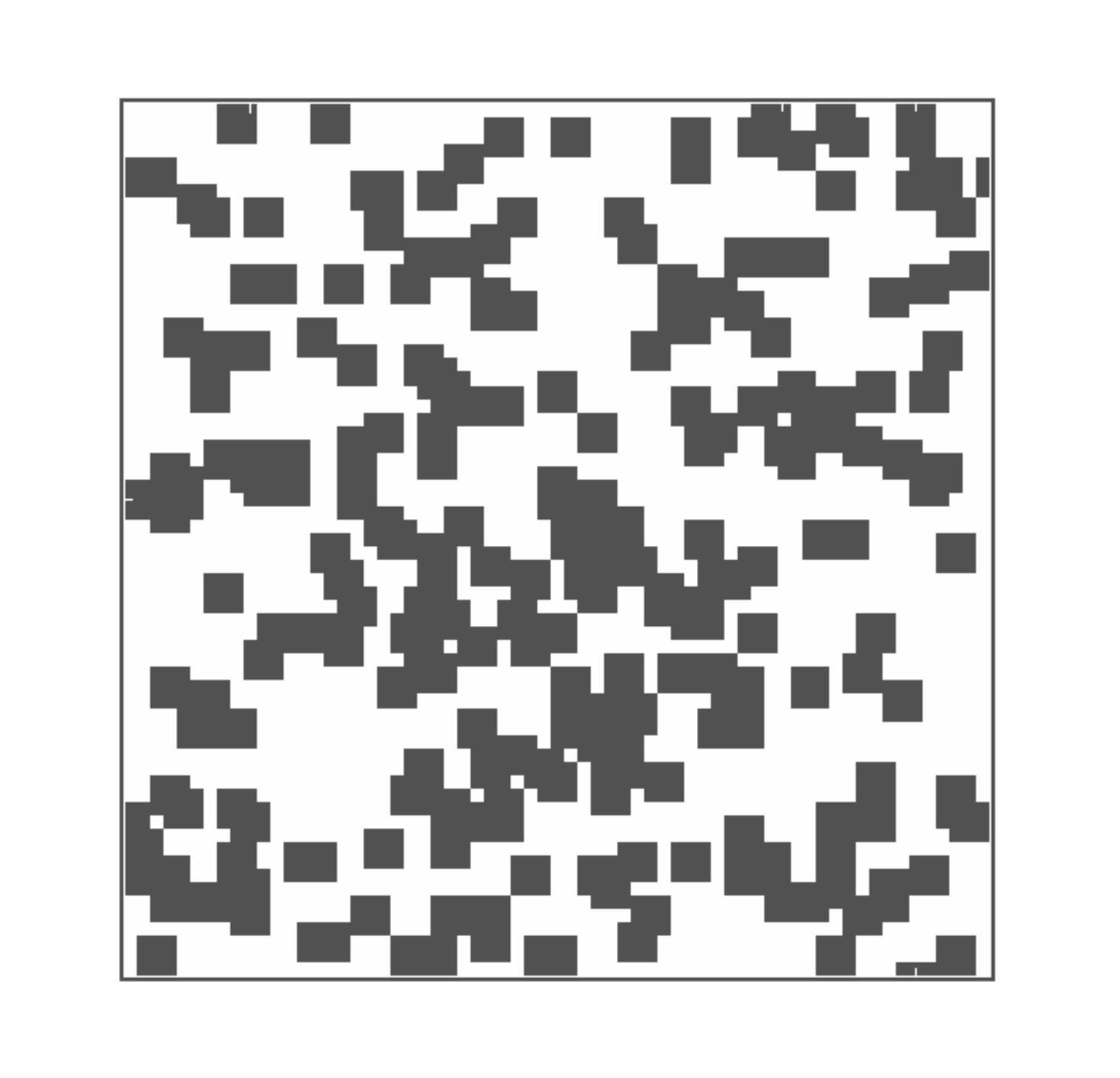}\\
\includegraphics[width=5.2cm]{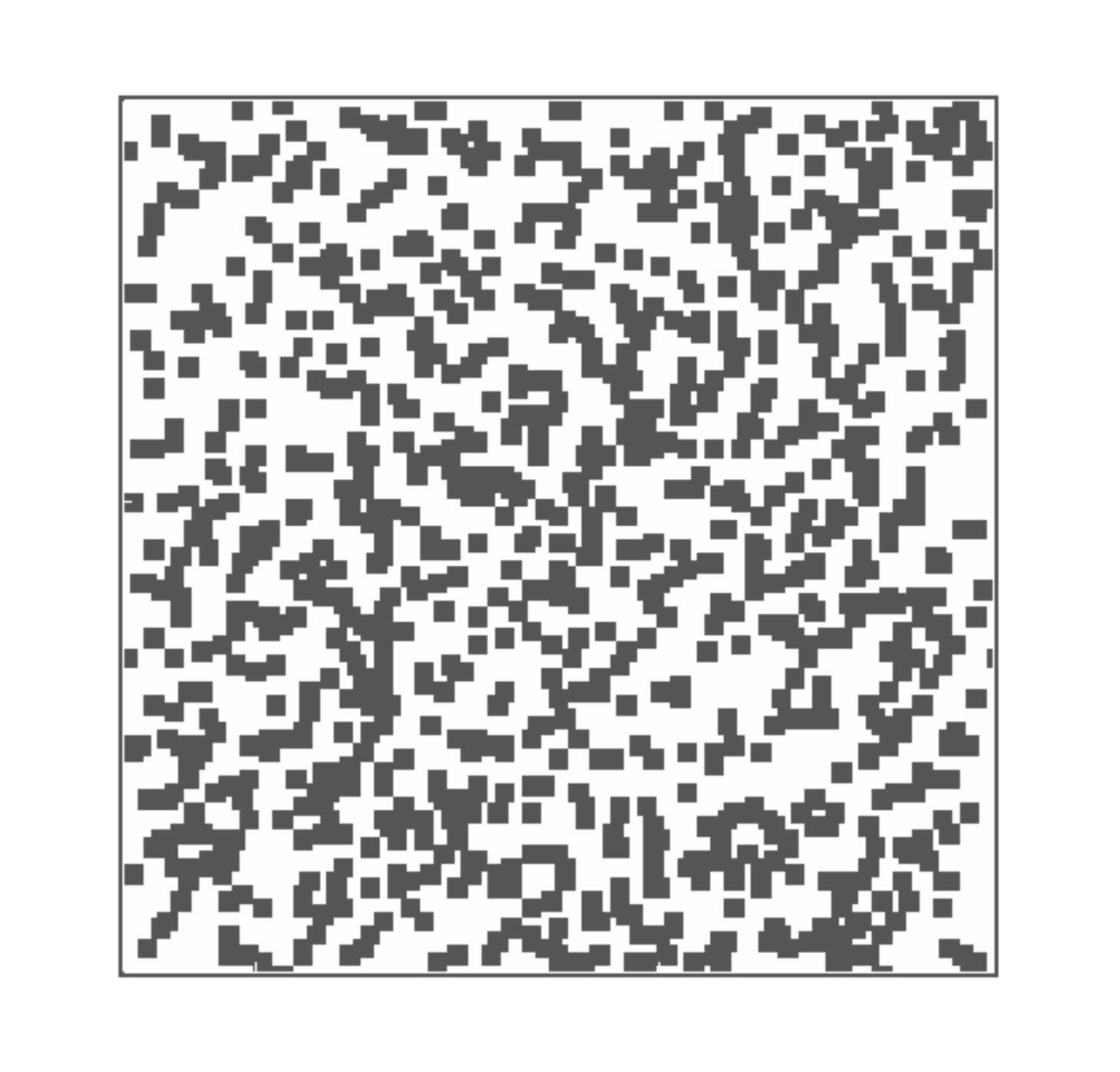}
\includegraphics[width=5.2cm]{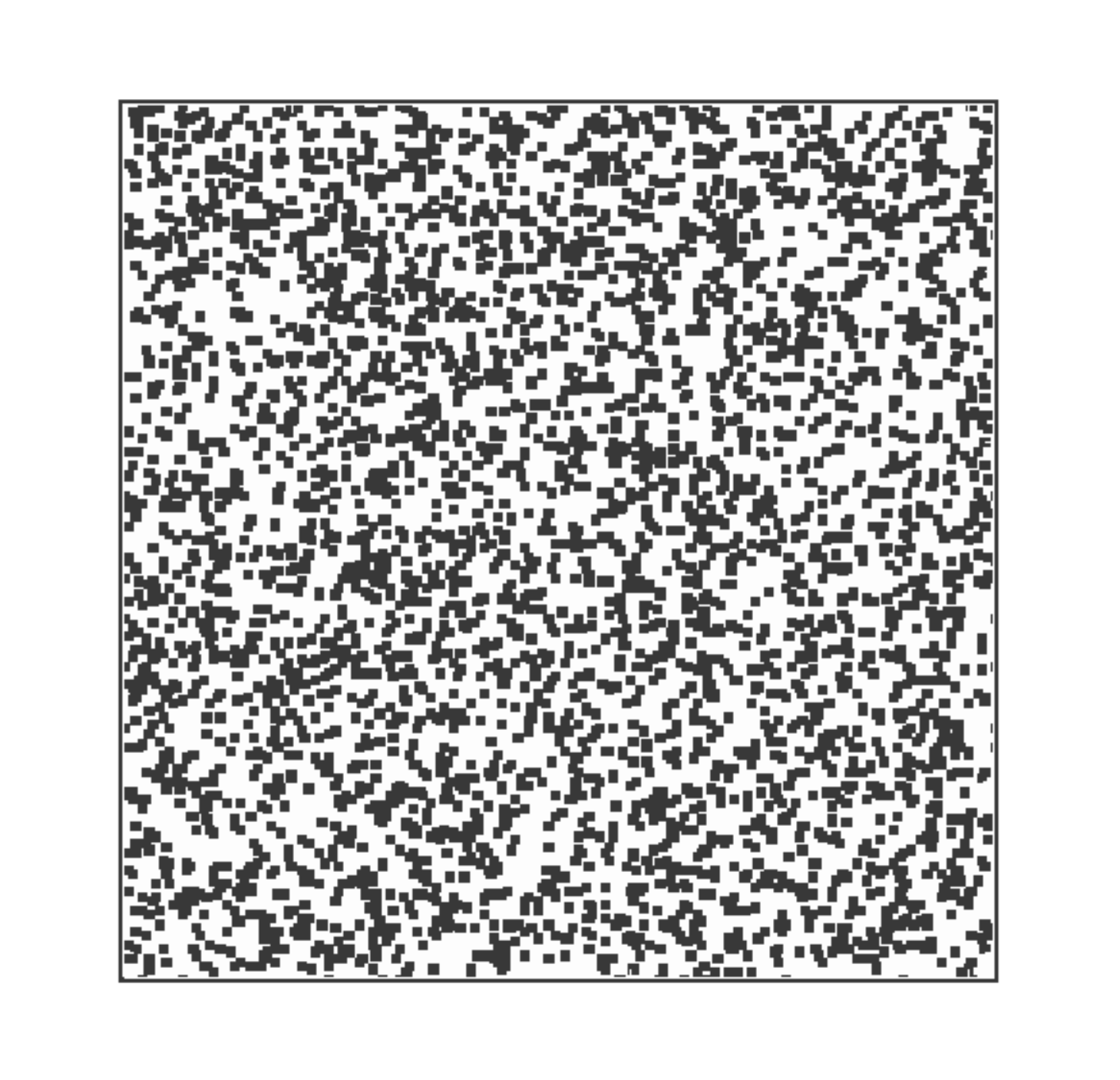}
\includegraphics[width=5.2cm]{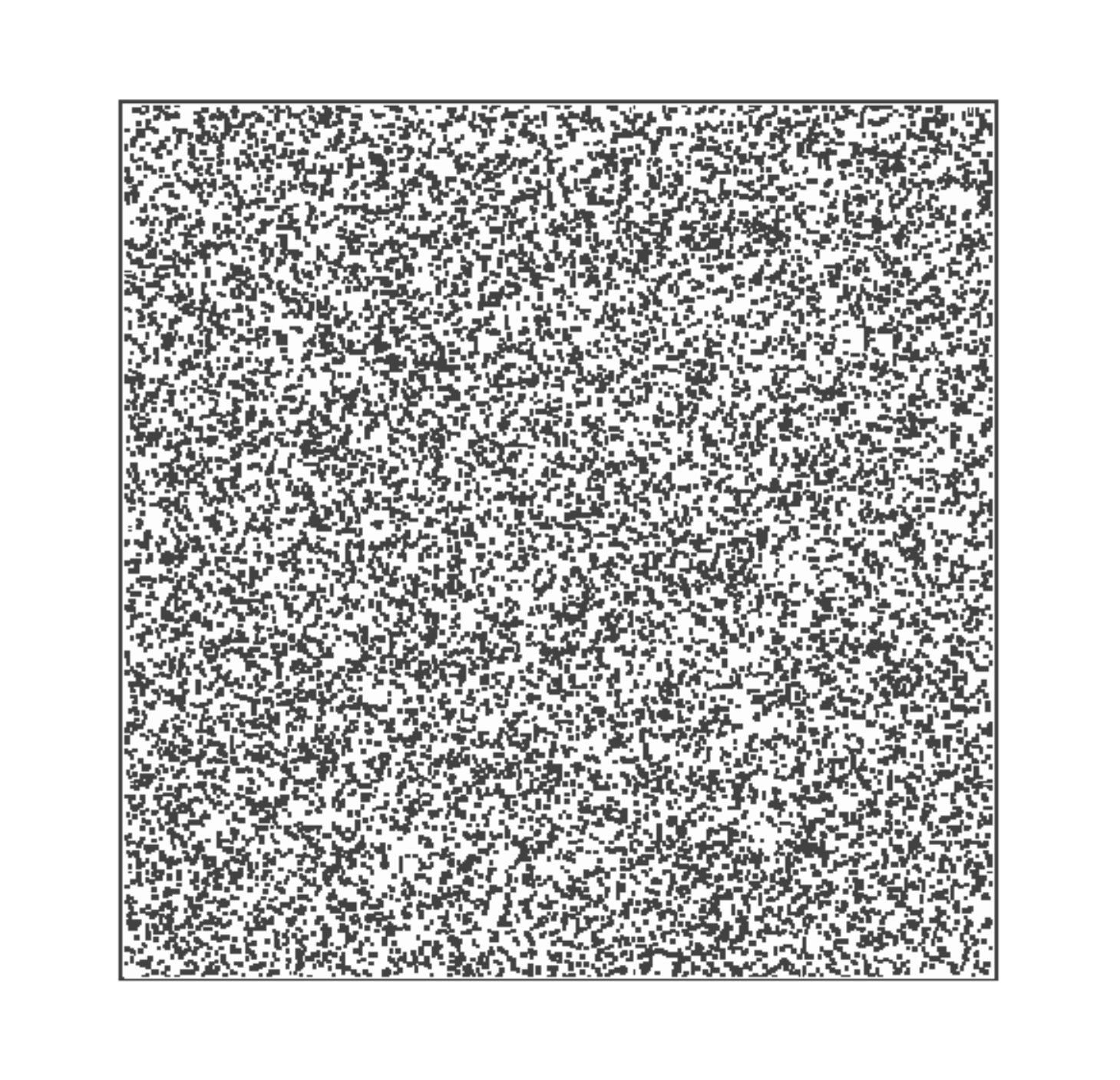}
\caption{\small Realization examples of a stochastic process with $L^2$ 
overlapping cells for $L=4,8,16$ (top) 
and $L=32,64,128$ (bottom), $\alpha=1/4$.}
\label{fig:Clust_vsL}
\end{figure}
Recall that the homogenization problem is solved in the unit square $\Omega=[0, 1]^2$ 
with the grid size $m_s\times m_s $, where $m_s = m_0 \, L + 1$. 
Due to tensor-based construction of the stiffness matrix and 
sparse representation of matrix entities, in our  numerical experiments using MATLAB, 
the largest number of generated homogenization cells in the domain $\Omega$ reaches the value 
up to $L^2=128^2$.  
It corresponds to the problem (vector) size  $263169$  ($m_s=513$ with $m_0=4$). 

Figure \ref{fig:Clust_vsL} illustrates examples of distributions of $L^2$ 
randomly located (overlapping) cells specifying the equation coefficient in the cases of 
moderate and large size of the \emph{representative volume elements} (RVE) for $L=4,\, \ldots,\, 128$,
used in the study of asymptotic of  empirical variance/average versus the size of the  RVE, $L$.
 
\begin{table}[tbh]
\begin{center}\footnotesize
\begin{tabular}
[c]{|r|r|r|c|c|c|}%
\hline
$L^2$ &   $m$/$m^2$ &   matrix   & RHS    & PCG time   \\
\hline  
$4^2$ &    17/289     & 0.012     & 0.01    &  0.006    \\
\hline
$8^2$ &    33/1089      & 0.06    & 0.045   & 0.137      \\
\hline
$16^2$ &   65/4225      & 0.34    & 0.19    & 0.11    \\
\hline
$32^2$ &   129/16641    & 3.0     & 0.8     &  0.5     \\
\hline
$64^2$ &   257/66049     & 36     & 3.7     & 2.6       \\
\hline
$128^2$ &  513/263169   & 561     & 22      & 13.8      \\
\hline
 \end{tabular}
\caption{ \small  CPU times (sec) versus the number of inclusions (i.e., $L^2$) 
for generating the stiffness matrix, the right-hand side, 
and for the solution of the discretized system for the case of overlapping inclusions.
Tolerance $\varepsilon=10^{-8}$.
}
\label{Tab:Times_2D}
\end{center}
\end{table}

  Table \ref{Tab:Times_2D} presents the CPU times for generating the stiffness matrix, the right-hand side (RHS), 
and for the solution of the discretized system for the case of overlapping inclusions, for 
tolerance $\varepsilon=10^{-8}$. Number of inclusions ($L^2$) varies from $16$ to $16384$.
The latter is computed on a mesh of size $513\times 513$.
We observe that matrix generation takes the dominating time.

 \subsection{Systematic error  and empirical variance versus  $L$}

In what follows, we numerically check the theoretical convergence rate (\ref{eqn:CLT_RVE}), 
in form of checking (\ref{eqn:var}) and (\ref{eqn:sysErr}) separately.
%

Figure \ref{fig:sys_a11_vsL} serves to illustrate the asymptotic convergence of the
systematic error  see (\ref{eqn:sysErr}), at the limit of large $L$. 
Since we do not have access to the ensemble averages $\langle \bar{\mathbb{A}}_L\rangle_L$, 
we take empirical averages $ \bar{\mathbb{A}}_L^{N} $ for large enough $N$, 
(cf. (\ref{eqn:A_hom})) as a proxy. Furthermore, 
due to the fact that $\mathbb{A}_{\mbox{\footnotesize hom}}$ is not computable 
we compare the differences in $\langle \bar{a}_{L,11}\rangle_L$ computed on
a sequence of increasing values of $L$.

\begin{figure}[htb]
\centering
\includegraphics[width=7.2cm]{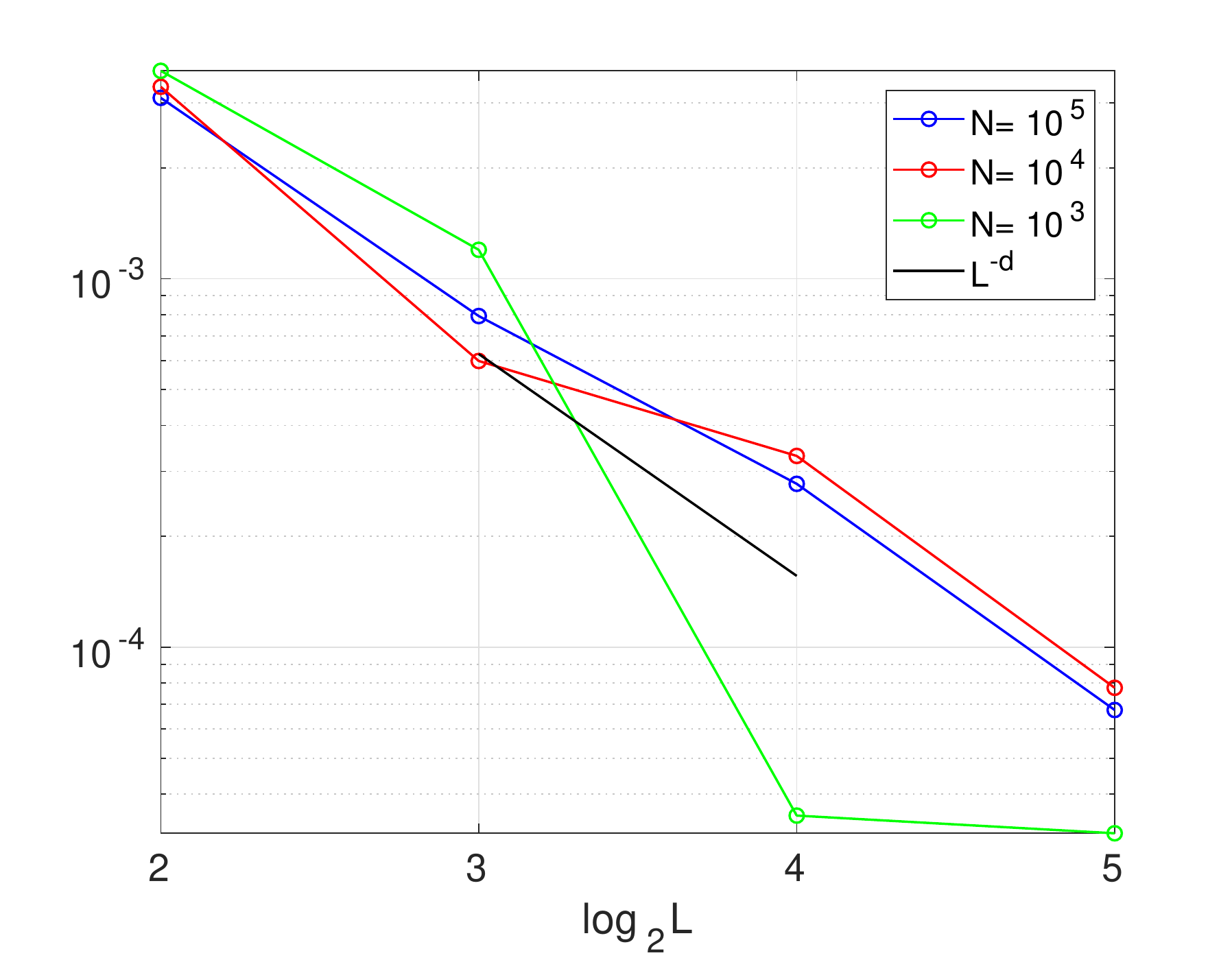}
 \caption{\small Systematic error 
 $\langle \bar{a}_{L,11}\rangle_L - \langle \bar{a}_{2L,11}\rangle_{2L}$  vs. $L$, 
 for increasing $L=2^p$, $ p=1,\; 2,\, \ldots, 5$ computed 
   for the largest number of realizations $N=10^5$. 
 }
\label{fig:sys_a11_vsL}
\end{figure}
 
Figure \ref{fig:sys_a11_vsL} shows the differences in matrix entries
$\langle \bar{a}_{L,11}\rangle_L - \langle \bar{a}_{2L,11}\rangle_{2L}$ 
for increasing sizes of the RVE, i.e., for $L=2^p$, $ p=1,\; 2,\, \ldots, 5$,  
  computed with $N=10^{5}$, $N=10^{4}$ and $N=10^{3}$  
stochastic realizations. 
 It  illustrates the asymptotic
convergence of the \emph{systematic error}, see (\ref{eqn:sysErr}),
\[  
  \left| \langle \bar{\mathbb{A}}_L \rangle_L - 
  \mathbb{A}_{\mbox{\footnotesize hom}}\right|  \lesssim L^{-d} \log^d L
\]
at the limit of large $L$.    Calculations are performed with $m_0=4$,  
 $\alpha=\frac{1}{4}$ and $\lambda=0.4$ and tolerance $\varepsilon=10^{-8}$. 
The black line corresponds to the curve $L^{-d}$, with $d=2$.   

  The largest size of RVE  
with $p=\log_2 L =5$, presented in statistics in Figure \ref{fig:sys_a11_vsL}, 
corresponds to the most left picture in the bottom row in Figure \ref{fig:Clust_vsL}.
In this example the jumping coefficient contains $32^2$ (overlapping) inclusions, 
and the discrete problem of size $m_s^2=129^2$ (i.e., vector  size is $16641$) 
has been solved $N=10^5$ times for providing the representative statistics.  
For readers convenience, Table \ref{Tab:a_sys_er_L1} presents the same data
 visualized in Figure \ref{fig:sys_a11_vsL}. 
 
\begin{table}[htb]
\begin{center}%
{\footnotesize
\begin{tabular}
[c]{|r|r|r|r| }%
\hline
 L / N &    $ 10^5 $ &  $10^4$ & $10^3$ \\
 \hline
 4  & 0.003095    & 0.003316  & 0.003665  \\
  \hline
 8 &  0.000792    & 0.000598  & 0.001198  \\
     \hline
 16 &  0.000277   & 0.000330  & -0.000034  \\
      \hline
32 &   0.000067   & 0.000077  &  0.000031  \\
  \hline
  \end{tabular}
\caption{\small Systematic error $\langle \bar{a}_{L,11}\rangle_L - \langle \bar{a}_{2L,11}\rangle_{2L}$  vs. $L$, 
 for increasing $L=2^p$, $ p=1,\; 2,\, \ldots, 5$, computed for $N=10^5,10^4$ and $10^3$ 
realizations. 
}
\label{Tab:a_sys_er_L1}
}
\end{center}
\end{table}
 
We now turn to the random error, i.e., the variance of $\langle \bar{\mathbb{A}}_L \rangle_L = [\bar{a}_{ij}]$.
Since by symmetry considerations, 
$\langle \bar{a}_{11} \rangle_L = \langle \bar{a}_{22}\rangle_L$ and
$\langle \bar{a}_{12}\rangle_L =0$, we monitor 
$$
\langle(\bar{a}_{L,11} - \bar{a}_{L,22})^2\rangle^{1/2}_L \quad  \mbox{and} \quad
\langle(\bar{a}_{L,12})^2\rangle^{1/2}_L,
$$
which should decay as $1/L$.
Again, since we do not have access to the ensemble averages defining the standard deviation, 
we replace them by their empirical approximation for large enough $N$,
 \begin{align}\label{eqn:Var1}
  \left( \frac{1}{N}\sum\limits_{n=1}^N (\bar{a}_{L,12}^{(n)})^2\right)^{1/2}
   \approx \; &
\langle(\bar{a}_{L,12})^2\rangle^{1/2}_L
   \leq C_1 L^{-d/2}, \\
   \label{eqn:Var2}
\left(\frac{1}{N}\sum\limits_{n=1}^N (\bar{a}_{L,11}-\bar{a}_{L,22})^2\right)^{1/2}
  \approx \; & \langle(\bar{a}_{L,11} -\bar{a}_{L,22})^2\rangle^{1/2}_L  \leq C_1 L^{-d/2}.
 \end{align}
 
 \begin{figure}[htb]
\centering
\includegraphics[width=7.2cm]{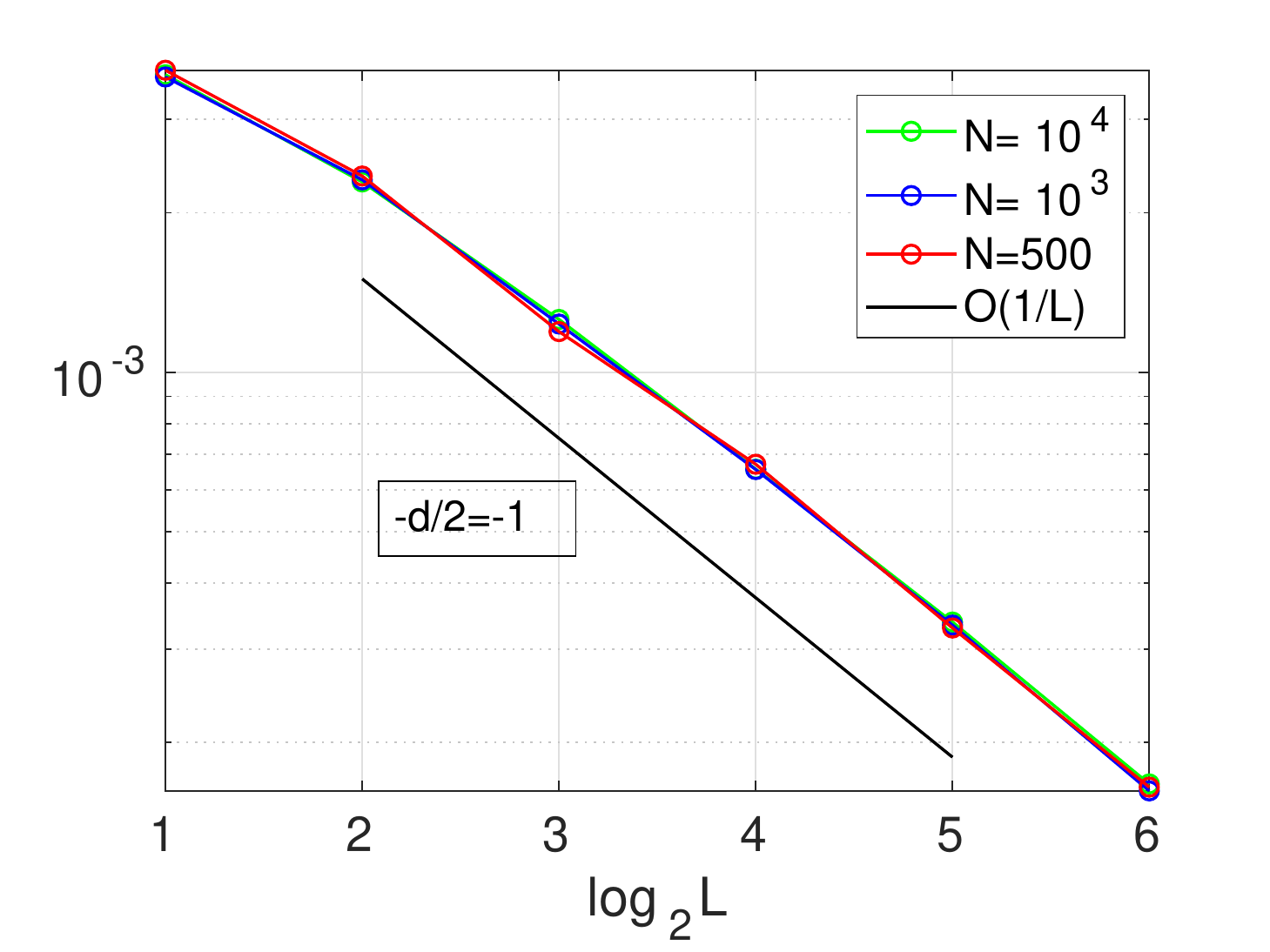}
 \includegraphics[width=7.2cm]{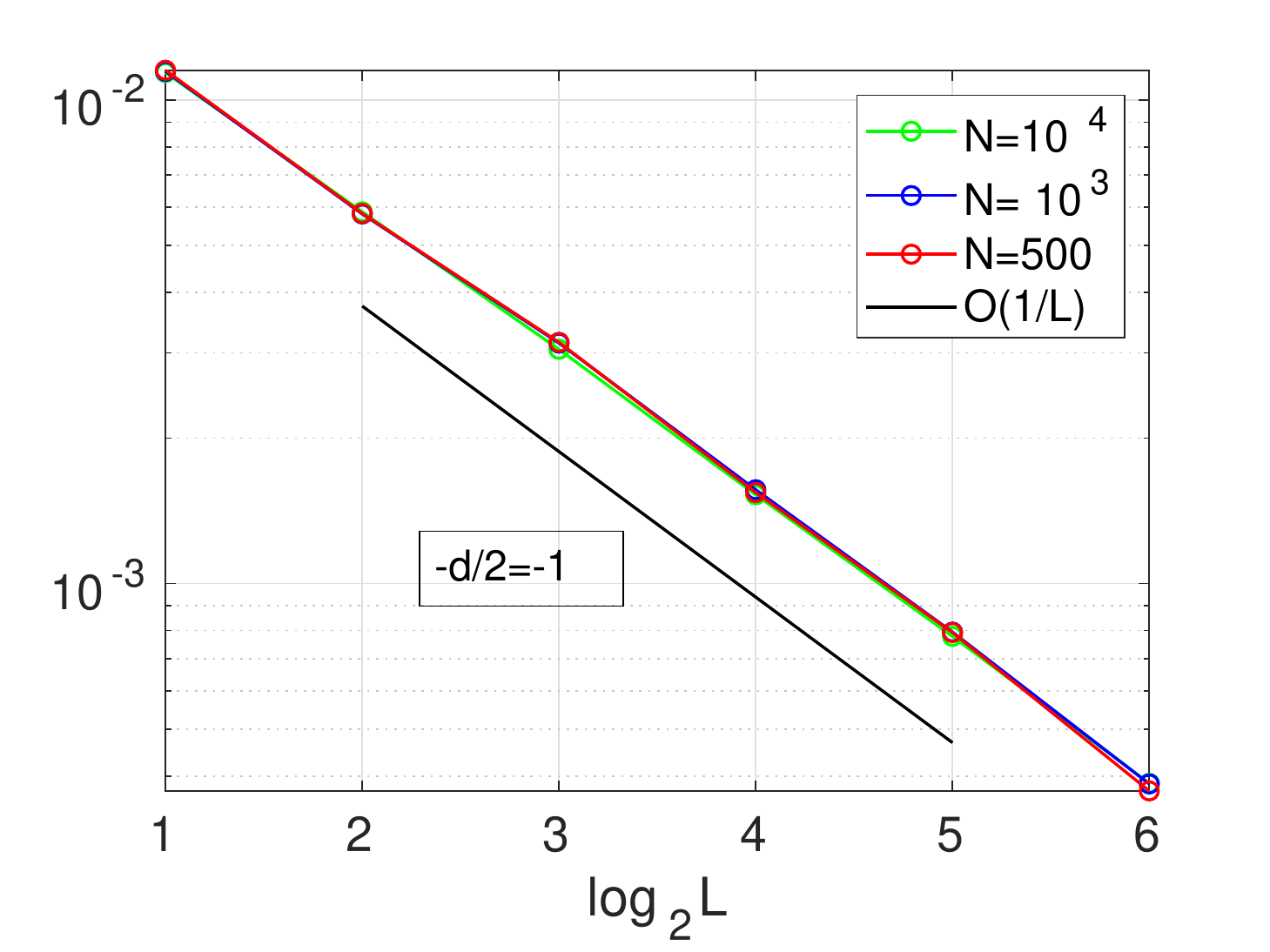}
\caption{\small Standard deviation of $\bar{a}_{L,12}$ (left) and of
$(\bar{a}_{L,11} - \bar{a}_{L,22})$ (right) versus $L$, with 
$L=2^p$, $ p=1,\; 2,\, \ldots, 6$,  for $N=500, 10^3$ and $10^4$.}
\label{fig:EmpEvarA12_vsL}
\end{figure}
Figure \ref{fig:EmpEvarA12_vsL} presents the empirical average (standard deviation) 
 for  $\bar{a}_{L,12}$ and $\bar{a}_{L,11} - \bar{a}_{L,22}$ vs.
$L=2,\, 4,\, \ldots,\, 64$, corresponding to $N=500, 10^3$ and $10^{4}$ 
realizations ($\alpha=\frac{1}{4}$, $\lambda=0.4$), confirming the estimates (\ref{eqn:Var1}) and
(\ref{eqn:Var2}).  Notice that starting from $N=500$ the values of empirical average
 for different number of realizations practically coincide.
The results  for the homogenized matrix for the largest $p=6$ presented in 
Figure \ref{fig:EmpEvarA12_vsL} correspond to ensembles with $4096$ overlapping cells, 
and the size of the discrete problem  (i.e., vector/matrix size) is $66049$.
These systems of equations have been solved $10^4$ times. 
An example of realization for $p=6$ is shown in Figure \ref{fig:Clust_vsL} (middle bottom panel). 
   \begin{table}[htb]
\begin{center}%
{\footnotesize
\begin{tabular}
[c]{|r|r|r||r|r| }%
\hline
   & \multicolumn{2}{|c|}{$\bar{a}_{L,12}$} & \multicolumn{2}{|c|}{$\bar{a}_{L,11}- \bar{a}_{L,22}$} \\
   \hline
 L / N & $ 10^4 $ &  $500$ & $10^4$ & $500$ \\
 \hline 
   2  & 0.003643    & 0.003716  & 0.011402 & 0.011531  \\
  \hline
 4  & 0.002287    & 0.002346  & 0.005875 & 0.005828  \\
  \hline
 8 &  0.001258    & 0.001193  & 0.003052 &  0.003156 \\
     \hline
 16 &  0.000656   & 0.000670  & 0.001527 &  0.001543 \\
      \hline
 32 &  0.000337   & 0.000329  & 0.000778 &  0.000792 \\
 \hline
 64 &  0.000167   & 0.000165  & 0.000386 &  0.000372 \\
  \hline
  \end{tabular}
\caption{\small Standard deviation of $\bar{a}_{L,12}$ (left) and
$\bar{a}_{L,11} - \bar{a}_{L,22}$ (right) versus $L$, with 
$L=2^p$, $ p=1,\; 2,\, \ldots, 6$,  for $N=500$ and $N=10^4$.} 
\label{Tab:a_stan_dev_L1}
}
\end{center}
\end{table}
 
 Related to Figure \ref{fig:EmpEvarA12_vsL},
Table \ref{Tab:a_stan_dev_L1} presents standard deviation of $\bar{a}_{L,12}$ (left) and
$\bar{a}_{L,11} - \bar{a}_{L,22}$ (right) versus $L$, with 
$L=2^p$, $ p=1,\; 2,\, \ldots, 6$,  for $N=500$ and $N=10^4$. 
We choose the following discretization and model parameters
 $m_0=4$,  $\varepsilon=10^{-8}$, $\alpha=\frac{1}{4}$, and $\lambda=0.4$.
 
 We summarize that numerical results presented in Figures \ref{fig:sys_a11_vsL} and 
 \ref{fig:EmpEvarA12_vsL} (see also Tables \ref{Tab:a_sys_er_L1} and \ref{fig:EmpEvarA12_vsL})
 confirm the asymptotic convergence rates of the systematic error (\ref{eqn:sysErr}) and 
 the empirical average (\ref{eqn:Var1}), (\ref{eqn:Var2}) in the size $L$ of RVE.

\subsection{The asymptotic of quartic tensor vs. leading order variances}
\label{ssec:Quartic_tensQ}

In this section, we consider the convergence of the \emph{quartic tensor} $\bar{ Q}_L$, 
representing \emph{covariances} of the matrix $\bar{\mathbb{A}}_L$, to its 
leading order variances ${Q}_{\mbox{\footnotesize hom}}$, see Section \ref{ssec:Symmetr_Quartic_tensQ}. 
For the large number of realizations $N$,
the computable approximation, 
$\bar{ Q}_{L}^N \in \mathbb{R}^{2\times 2 \times 2 \times 2}$, 
to the scaled quartic tensor is defined by 
\begin{equation}\label{quartic}
 \bar{ Q}_{L}^N= \frac{L^d}{N -1}\sum_{n=1}^{N} (\bar{\mathbb{A}}^{(n)}_L - 
 \frac{1}{N}\sum_{n'=1}^{N} \bar{\mathbb{A}}^{(n')}_L )^{\otimes 2},
\end{equation}
so that by the central limit theorem
\[
 \bar{ Q}_L:= \lim\limits_{N \to \infty} \bar{ Q}_{L}^N.
\]
The equivalent matrix representation of $\bar{ Q}_{L}^N$ 
is obtained by setting the operation $\otimes 2$  in (\ref{quartic}) as the Kronecker product of matrices
(see Definition \ref{def:Kron}), 
further denoted by 
\[
 \bar{{Q}}^N_L=[\bar{q}_{L,ij}] \in \mathbb{R}^{4\times 4},\quad 
i,j=1,\ldots,4.
\]
In our numerical tests we shall check the asymptotic behavior 
\[
 \left\langle \left| \frac{L^d}{N-1}\sum_{n=1}^{N} (\bar{\mathbb{A}}^{(n)}_L - 
 \frac{1}{N}\sum_{n'=1}^{N} \bar{\mathbb{A}}^{(n')}_L )^{\otimes 2} - 
 {Q}_{\mbox{\footnotesize hom}} \right|^2  \right\rangle_L  \lesssim L^{-d} \ln ^d L,
\]
which can be expected at the limit of large size $L$ of the RVE, see \cite{GlNeuOtto:13,GlOtto_Fluct:17}.
\begin{figure}[htb]
\centering
\includegraphics[width=7.2cm]{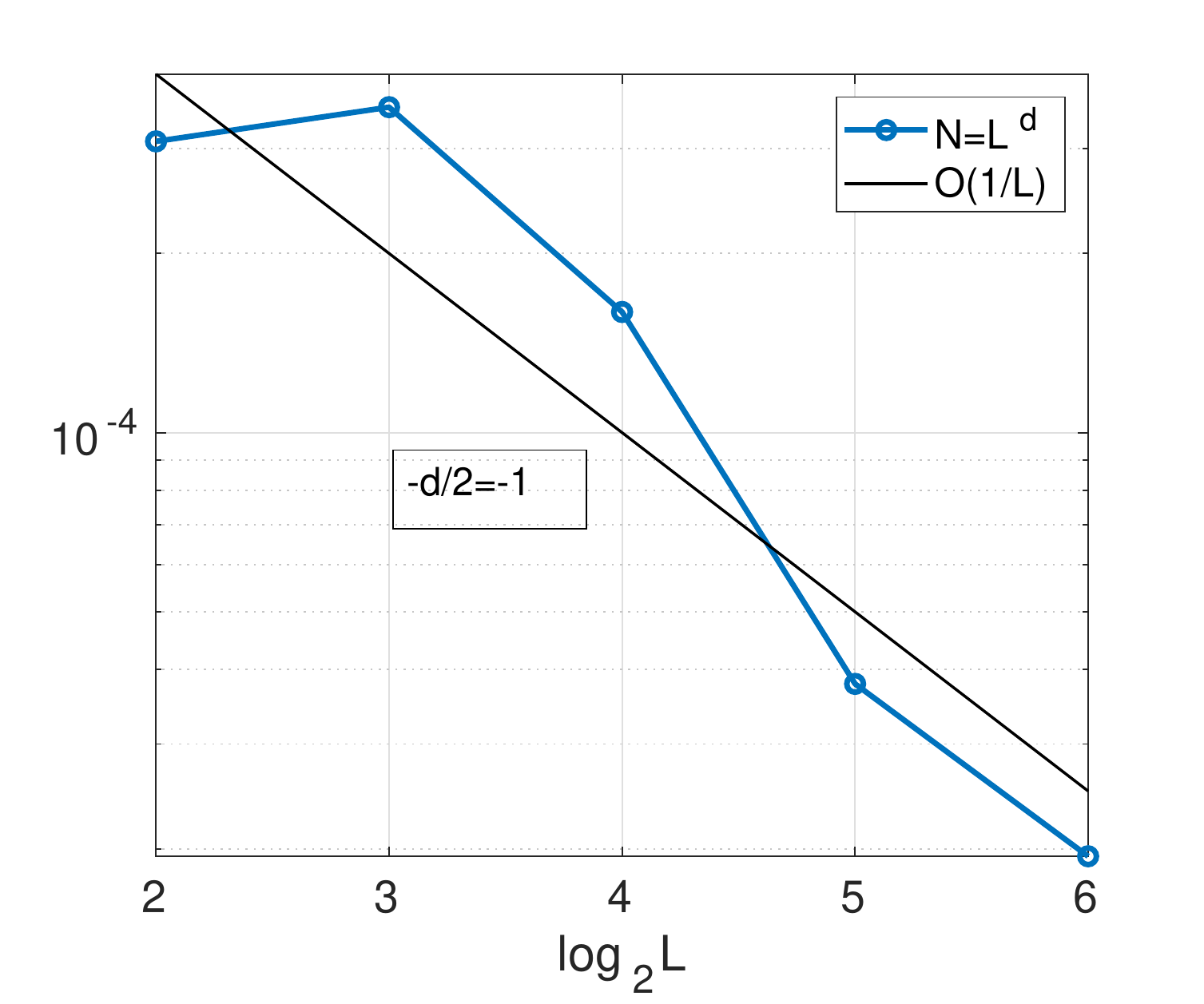}
\includegraphics[width=7.2cm]{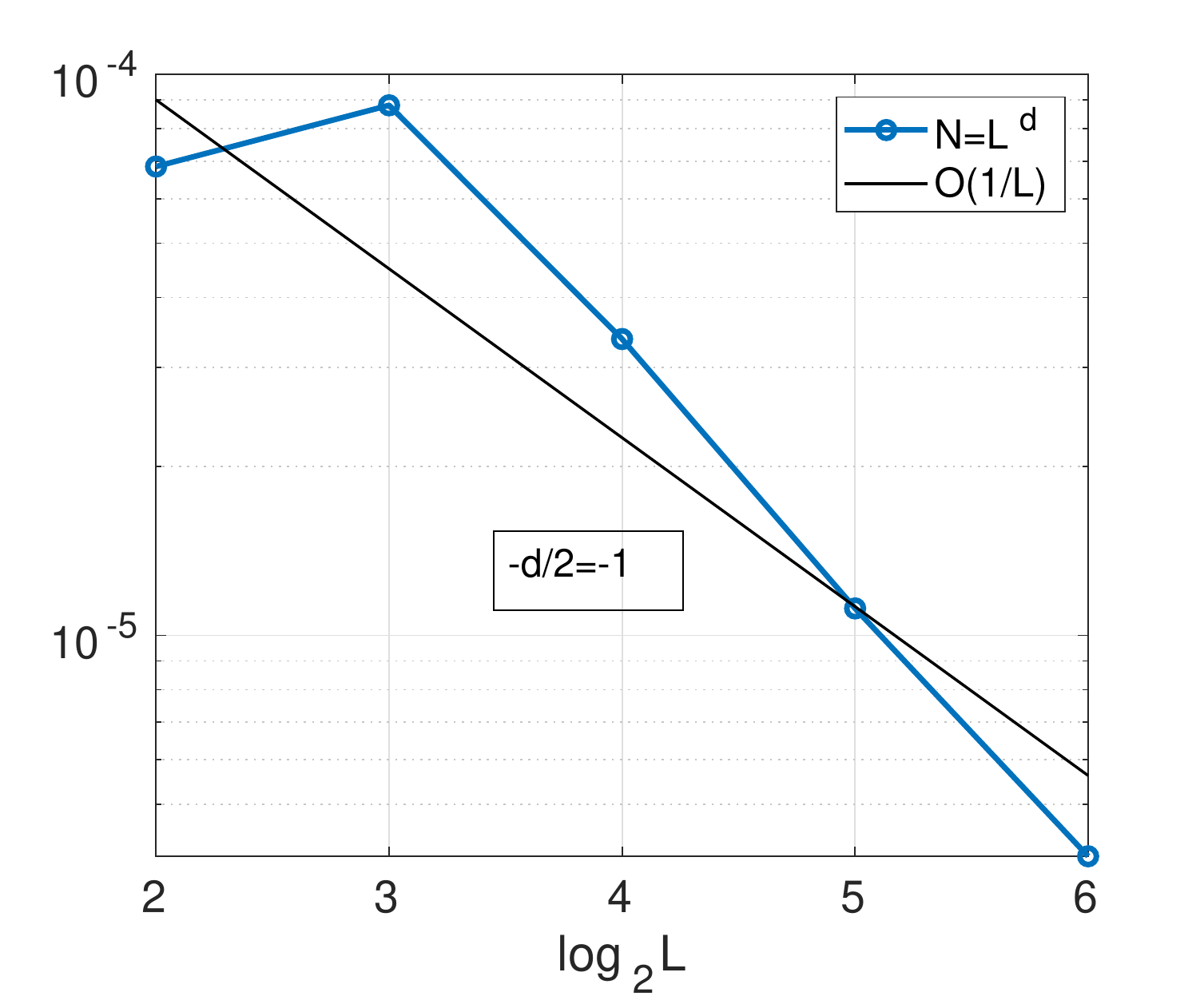}
\caption{\small 
$\bar{q}_{L,11}-\bar{q}_{2L,11}$  (left) and $\bar{q}_{L,14}-\bar{q}_{2L,14}$ (right) 
versus $L=2^p$, $p=1,\ldots, 6$, $N=L^d$.}
\label{fig:DevQuarticTensor_vsL}
\end{figure}
%

It is worth to note that the quartic tensor $\bar{ Q}_L$
can be calculated at no further cost than the effective homogenized matrix 
$\bar{\mathbb{A}}_{L}$.
%

Figure  \ref{fig:DevQuarticTensor_vsL} 
shows the diagonal elements,
$\bar{q}_{L,11} - \bar{q}_{2L,11}$ and $\bar{q}_{L,14} - \bar{q}_{2L,14}$, in
quartic tensor $\bar{ Q}_{L}^N$ versus increasing size of RVE $L=2^p$, $p=1,\ldots, 6$, 
(see (\ref{quartic})), for $N=L^d$.
Figure  \ref{fig:DevQuarticTensor_vsL} confirms convergence rate of 
$\bar{q}_{L,11} - \bar{q}_{2L,11}$ in RVE $L$ as  $O(L^{d/2}$.
 

\section{Conclusions}\label{sec:concl}

We present the numerical scheme for discretization and solution of 2D elliptic equations 
with strongly varying piecewise constant coefficients arising in stochastic homogenization
of multiscale composite materials.
The resulting large linear system of equations is solved by the preconditioned  
CG iteration  with  the convergence rate that is independent  of  the grid size and of  
the variation in jumping coefficients.
For a fixed size of the representative volume element, our approach allows to 
avoid the generation of the new FEM space in each stochastic realization.
For every realization, fast assembling of the FEM 
stiffness matrix is performed by agglomerating the Kronecker tensor 
products of 1D FEM discretization matrices. The
resultant stiffness matrix is maintained in a sparse matrix format.
%
 
 Our numerical scheme allows to investigate the asymptotic convergence rate 
of significant quantities of
stochastic homogenization process in the course of a large number of realizations 
(of the order of $N =10^5$) and for large sizes
of the representative volume elements up to $L=128$, corresponding to the 
number of inclusions $16384$ and matrix size $513^2\times 513^2$.   
Note that for every realization a new matrix generation and solution of the respective 
linear system is performed.
 
Our numerical experiments study the asymptotic convergence rate  
of systematic error and standard deviation 
in the size of RVE, rigorously established in \cite{GlNeuOtto:13}.
In particular, we confirm in various numerical tests the theoretical asymptotic estimates, 
see Section \ref{ssec:SystErStandDev}, concerning 
the convergence rate $O(1/L)$ for the empirical variance at the limit of large $L$, but with 
a moderate number of stochastic realizations $N$, and the asymptotic 
$C L^{-2}\ln^2 L$ in the case of large $N$.


The asymptotic behavior of covariances of the homogenized matrix in the form of quartic 
tensor are studied numerically.
In particular, we consider the asymptotic of the quartic tensor versus the leading 
order variances, computed for the large number of stochastic realizations up to $N=10^4$. 
In this way, the asymptotic $O(L^{-d} \ln^d L)$, for $d=2$, is confirmed on
a sequence of increasing sizes of the RVE, up to $L=64$.

The stochastic characteristics of the system are analyzed for a  
range of intrinsic model parameters like the number of realizations, 
the size of periodic representative volume element, the jump-ratio in the stochastic 
equation coefficients (contrast) and various grid discretization parameters.
The presented numerical scheme allows to perform large scale simulations using MATLAB 
on a moderate computer cluster.
The  tensor-based numerical techniques to matrix generation presented in this paper
can be extended to 3D and higher dimensional problems.

\section{Appendix: spectral density of a stochastic operator}\label{app:DOS_shomo} 

 Spectral properties of the randomly generated family of elliptic operators
 are important in many applications, in particular,
 in stochastic homogenization of time dependent PDEs. 
In what follows, we  analyze numerically the average behavior of the so-called density of spectrum (DOS) 
for the family of stochastically generated 2D elliptic operators $\{A_n\}$ for 
the large sequence of stochastic realizations $n=1,\ldots,N$. 
The DOS provides the important spectral characteristics of the stochastic differential operator
which accumulate the significant information on the static and dynamical characteristics of the complex 
physical or molecular system. Here we numerically demonstrate the convergence  of 
DOS to the sample average function at the limit of large number of stochastic realizations with fixed $L$.
Our second goal is the numerical study of the DOS depending on the increasing number $L$. 

\begin{figure}[htb]
\centering
\includegraphics[width=7.0cm]{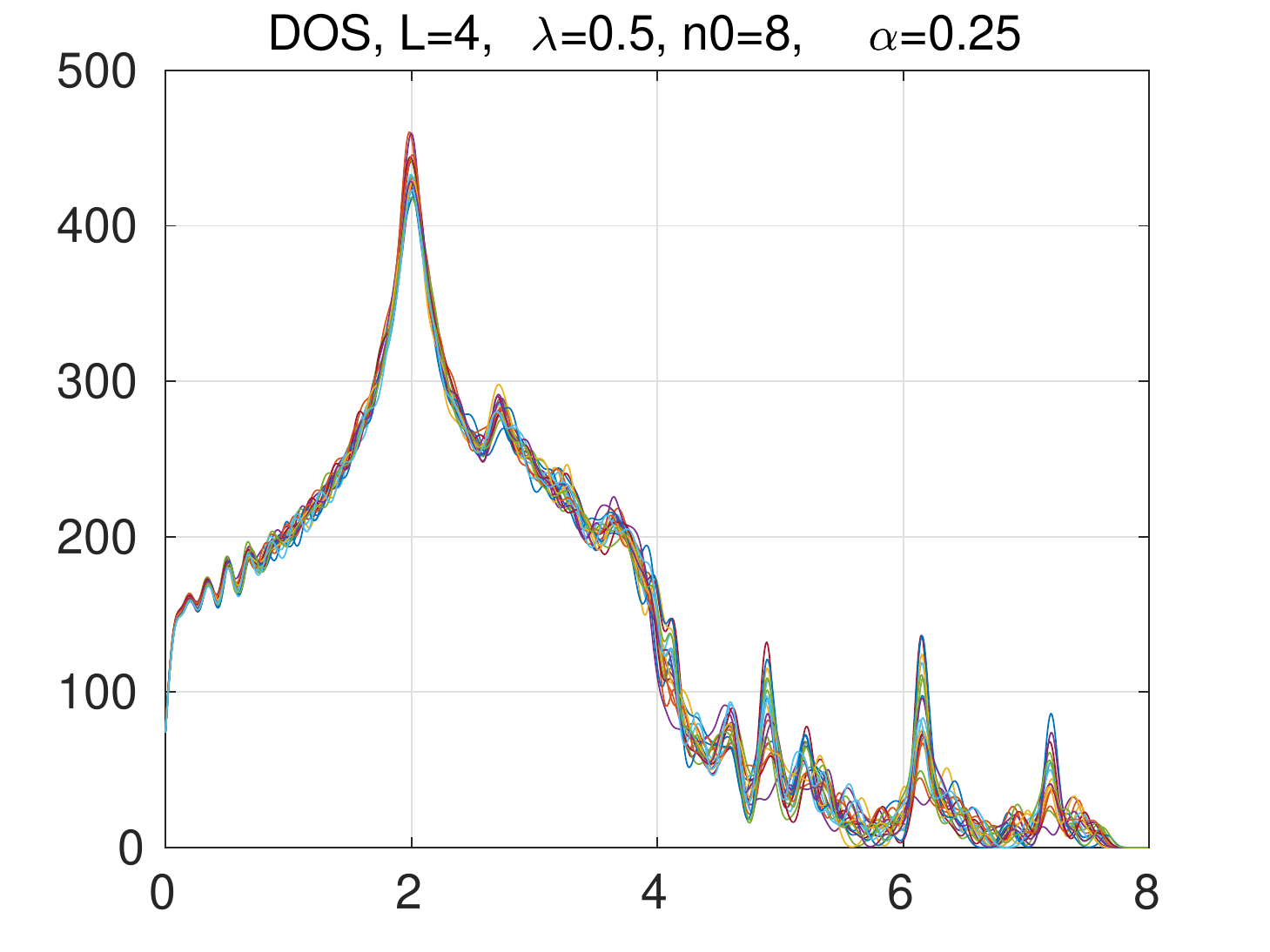}\quad
\includegraphics[width=7.0cm]{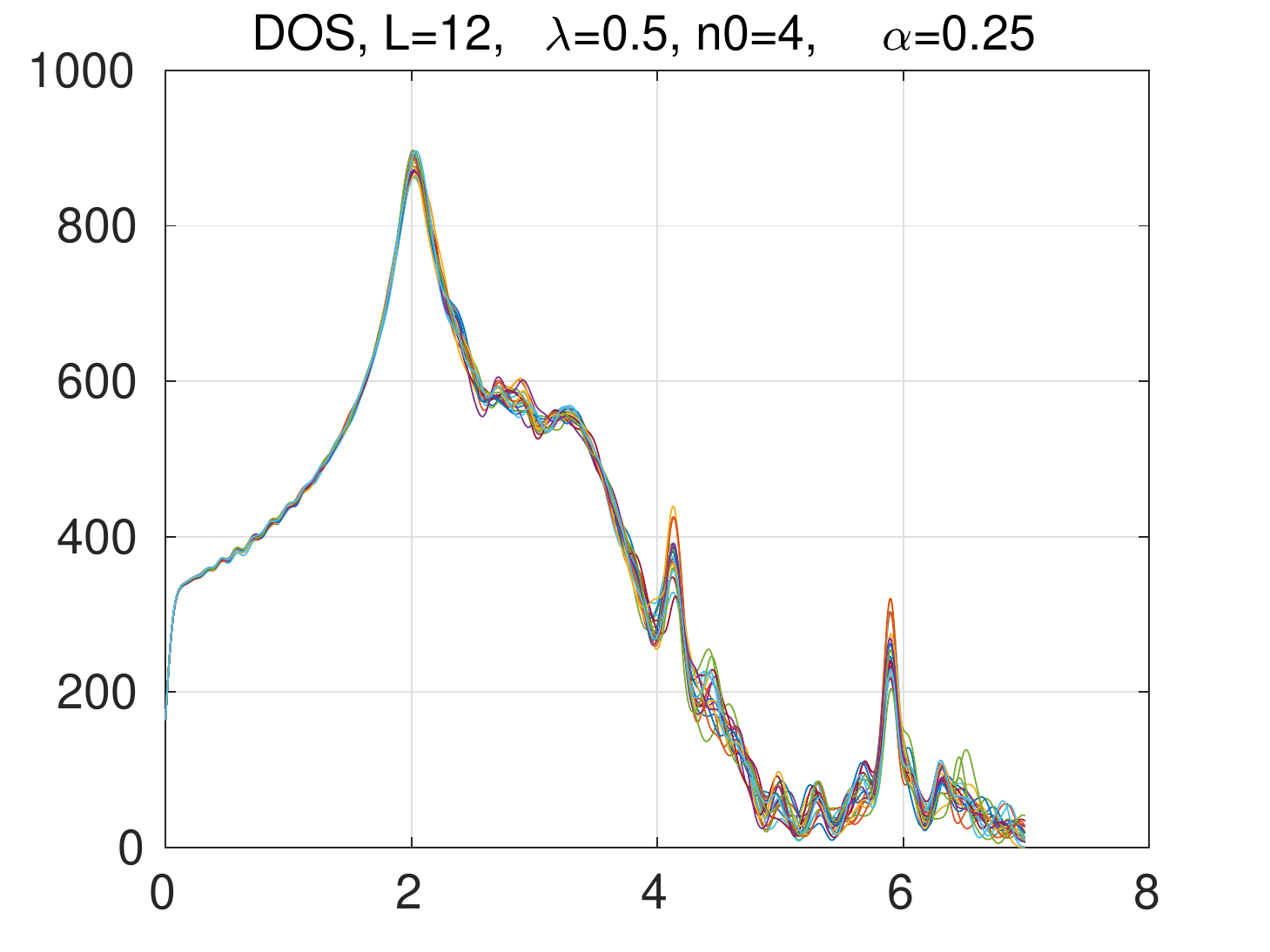}
\caption{\small Density of states for a number of stochastic processes $N=1, \;2, \ldots ,\; 20$  
with $L=4$ (left) and $L=12$ (right) for $\lambda =0.5$, $n_0=8$ and $\alpha = 0.25$.}
\label{fig:DOS_discr_L4L12}
\end{figure}  
We use the simple definition of DOS for symmetric matrices, see 
\cite{LinSaadYa:16,BeKhKhYa_AbsorpSp:16},
\begin{equation} \label{eqn:DOS}
 \phi(t)= \frac{1}{m}\sum\limits_{j=1}^{m} \delta(t-\lambda_j),\quad t,\lambda_j\in [0,a],
\end{equation}
where $\delta$ is the Dirac function and the $\lambda_j$'s are the eigenvalues 
of $A=A^T\in \mathbb{R}^{m\times m}$, 
\[
 A {\bf u}_j= \lambda_j {\bf u}_j, \quad j=1,\ldots,m,
\]
assumed to be labeled non-decreasingly. 

In the presented analysis we employ  the commonly used class \cite{LinSaadYa:16} 
of blurring approximations to the spectral density $\phi(t)$ 
by using regularization via a Gaussian function with width parameter $\eta>0$,
\[
\delta(t) \rightsquigarrow
 g_\eta(t)= \frac{1}{\sqrt{2 \pi}\eta}\exp{\left(-\frac{t^2}{2 \eta^2}\right)},
\]
where the choice of small regularization parameter $\eta$ depends on the particular problem setting.
The DOS in (\ref{eqn:DOS}) will be approximated by Gaussian broadening,
\begin{equation} \label{eqn:DOS_gauss}
 \phi(t)\mapsto \phi_\eta(t):= \frac{1}{m} \sum\limits_{j=1}^{m} g_\eta(t -\lambda_j),
 \quad t\in [\lambda_{min},\lambda_{max}].
\end{equation}

Figure \ref{fig:DOS_discr_L4L12} represents DOS for a sequence of $N=1, \;2, \ldots ,\; 20,$  
stochastic realizations with $L=4, 8$ from left to right, corresponding to the fixed model parameters
$\lambda =0.5$, $\alpha = 0.25$ and $m_0=8$.  
This figure demonstrates the clustering  of DOS around the sample average 
on a  sequence of $N$ realizations.
This can be compared with DOS corresponding to the homogenized coefficient.
The  numerical experiments show that the spectrum of the stochastic
operator  strongly depends on  the parameter  $\alpha $ of the stochastic process. Moreover,
we come to the following practically important observation.
\begin{remark}\label{rem:DOS_Per_L}
Figure \ref{fig:DOS_discr_L4L12}  indicates that
 the DOS calculated with fixed model parameters, but for different values of $L$ 
 have practically identical shapes in the case of periodic boundary conditions.
 This stochastic property is the reminiscence  of the corresponding feature for the deterministic
 lattice structured systems in the periodic super-cell.
\end{remark}
Finally, we notice that the specific feature of the homogenized DOS is that 
the sample average differs substantially from DOS for the operator with the 
homogenized coefficient.

\vspace{0.6cm}
   
{\bf Acknowledgements.}
The authors would like to thank Julian Fischer (IST Austria, Wien) and Ronald Kriemann 
(MPI MiS, Leipzig) for useful discussions concerning the problem setting.

\begin{footnotesize}

\end{footnotesize}

\end{document}